\documentclass[11pt, oneside]{article} 
\usepackage{amsmath, amsthm, amssymb,geometry,tikz-cd,color,enumerate}
\usepackage{enumitem}
\usepackage{quiver}
\usepackage{mathrsfs}
\usepackage[title]{appendix}

\usepackage{tocloft}
\definecolor{mylinkcolor}{rgb}{0.05,0.05,0.4}
\definecolor{mycitecolor}{rgb}{0.1,0.1,0.8}
\definecolor{myurlcolor}{rgb}{0.1,0.1,0.8}
\usepackage{hyperref}
\hypersetup{colorlinks,linkcolor=mylinkcolor,citecolor=mycitecolor,urlcolor=myurlcolor, linktocpage}

\newtheorem{thm}{Theorem}[section]

\newtheorem{lem}[thm]{Lemma}
\newtheorem{cor}[thm]{Corollary}
\newtheorem{prop}[thm]{Proposition}

\theoremstyle{definition}
\newtheorem{defn}[thm]{Definition}
\newtheorem{rem}[thm]{Remark}
\newtheorem{eg}[thm]{Example}

\title{Probability monads as codensity monads}
\author{Ruben Van Belle\thanks{School of Mathematics, University of Edinburgh; ruben.van.belle@ed.ac.uk}}
\begin{document}
\maketitle
\begin{abstract}
We show from a categorical point of view that probability measures on certain measurable or topological spaces arise canonically as the extension of probability distributions on countable sets. We do this by constructing probability monads as the codensity monads of functors that send a countable set to the space of probability distributions on that set. On (pre)measurable spaces we discuss monads of probability (pre)measures and their finitely additive analogues. We also give codensity constructions for monads of Radon measures on compact Hausdorff spaces and compact metric spaces and for the monad of Baire measures on Hausdorff spaces.

A crucial role in these constructions is given by integral representation theorems, which we derive from a generalized Daniell-Stone theorem.
\end{abstract}
\tableofcontents
\section{Introduction}
A probability distribution on a countable set $A$ is defined as a function $p:A\to [0,1]$ such that $$\sum_{a\in A}p(a)=1.$$ For uncountable sets however, this definition is not good any more. Indeed, let $p:A\to [0,1]$ be a function on an uncountable set $A$ such that $\sum_{a\in A}p(a)=1$; then the support of $p$ is countable. But this means that $p$ is essentially a probability distribution on a countable set. Therefore, to express  probability distribution on sets such as $\mathbb{R}$ or $\mathcal{C}([0,\infty),\mathbb{R})$ we need a different definition. Using measure theory a definition of a \emph{probability measure} can be given by Kolmogorov's axioms. 

An endofunctor $\mathcal{G}$ on the category of measurable spaces $\textbf{Mble}$ can be defined by sending a measurable space $X$ to the measurable space of all probability measures on $X$. Using Dirac delta probability measures and integration, this endofunctor can be given a monad structure. This monad is called the \emph{Giry monad} \cite{giry}. Monads on categories of measurable and topological spaces that are similar to the Giry monad are referred to as \emph{probability monads} \cite{jacobs}.

The right Kan extension of a functor along itself, assuming that it exists, can be given a natural monad structure. The obtained monad is called the \emph{codensity monad} of that functor \cite{leinster}. We will show that probability monads can be constructed as the codensity monads of functors that send a countable set $A$ to the the space of all probability measures on $A$. This shows that probability measures arise naturally as the categorical extension of the more intuitive probability measures on countable sets. We will discuss this construction for several monads of different kinds of probability measures on different kinds of measurable and topological spaces.

We begin the paper with the definition of codensity monad (section \ref{codensitymonads}). The rest of the paper can be divided in two parts. In the first part we will discuss probability monads on categories of measurable spaces. For completeness we start in section \ref{premeasurablespaces} with a detailed overview of probability measures and integration, which will lead to an integral representation theorem for probability measures. After this we will construct the Giry monad and variations of this monad as codensity monads in section \ref{probabilitymonadsoncategoriesofpremeasurablespaces}. Here the integral representation theorem from before will play a key role.
The second part is similar to the first part, only here we will talk about probability monads on categories of topological spaces. Again we start with an overview of probability measures and integration on topological spaces in section \ref{Hausdorffspaces} and give a slightly more general version of the Daniell-Stone representation theorem in section \ref{5}. This result will allow us to construct the Radon monad, bounded Lipschitz monad and Baire monad as codensity monads in section \ref{probabilitymonadsoncategoriesofHausdorffspaces}. In this section we will also briefly discuss another probability monad and the problems that arose in the attempt to construct this monad as a codensity monad.

An overview of the results in this paper is given in the following table.
\begin{center}
\begin{tabular}{ |c|c|c|c|c| } 
 \hline
Category & Probability monad & Theorem\\ 
 \hline
 measurable spaces & Giry monad & \ref{Girymonad}\\ 
measurable spaces & Giry monad of fin. add. probability measures & \ref{finadd}\\ 
premeasurable spaces & Giry monad of probability premeasures &\ref{premeasmonad}\\
 compact Hausdorff spaces & Radon monad &\ref{radonmonad}\\
 compact metric spaces & Bounded Lipschitz monad & \ref{blip}\\
 Hausdorff spaces & Baire monad & \ref{baire}\\
 \hline
\end{tabular}
\end{center}
\textbf{Related work}: The Giry monad was introduced by Giry in \cite{giry}. This monad and the finitely additive variation of it already have been constructed as codensity monads of inclusion functors of certain convex spaces in the category of measurable spaces by Avery in \cite{Avery}. The constructions we will present are different but related to Avery's constructions. Avery uses powers of the unit interval and affine maps, while we will use finite and countable simplices. The Radon monad originates from work by Semadeni \cite{semadeni} and Swirszcz \cite{swirszcz} and is further discussed in \cite{keimel} and \cite{jacobs}. The bounded Lipschitz monad is similar to the Kantorovich monad on the category of compact metric spaces, which was introduced by van Breugel in \cite{breugel}. This monad was extended to a monad on complete metric spaces by Perrone and Fritz \cite{Fritz1, Fritz2}.

\textbf{Acknowledgements}: I would like to thank Tom Leinster for helpful discussions. I would also like to thank the anonymous referee for suggesting to decompose the functor $G:\textbf{Set}_c\to \textbf{Mble}$ as $\textbf{Set}_c\xrightarrow{j}\textbf{Mble}\xrightarrow{\mathcal{G}}\textbf{Mble}$ and similar decompositions in the other sections. This led to several simplifications in the proofs and constructions.
\section{Codensity monads}\label{codensitymonads}
Every functor $G:\mathcal{D}\to \mathcal{C}$ that has a left adjoint $F:\mathcal{C}\to \mathcal{D}$ induces a monad on $\mathcal{C}$, namely the endofunctor $GF:\mathcal{C}\to \mathcal{C}$ together with the unit of the adjunction as unit of the monad and $G\epsilon_F$ as the multiplication of the monad. Suppose now that we have a functor $G:\mathcal{D}\to \mathcal{C}$ such that the right Kan extension of $G$ along $G$ exists; then $G$ still induces a monad on $\mathcal{C}$. Let $(T^G:\mathcal{C}\to \mathcal{C},\gamma:T^G\circ G\to G)$ be the right Kan extension of $G$ along $G$. 

There is a natural transformation $1_G:1_{\mathcal{C}}\circ G\to G$ and therefore by the universal property of right Kan extensions, there exists a unique $\eta:1_\mathcal{C}\to G$ such that the following diagrams are equal:

\[\begin{tikzcd}
	\mathcal{D} & {} & \mathcal{C} \\
	{} & \mathcal{C} & {}
	\arrow["G"', from=1-1, to=2-2]
	\arrow["{1_{\mathcal{C}}}"', from=2-2, to=1-3]
	\arrow[""{name=0, anchor=center, inner sep=0}, "G", from=1-1, to=1-3]
	\arrow["{1_G}", shorten <=5pt, shorten >=6pt, Rightarrow, from=2-2, to=0]
\end{tikzcd}
=
\begin{tikzcd}
	\mathcal{D} & {} & \mathcal{C} \\
	{} & \mathcal{C} & {}
	\arrow["G"', from=1-1, to=2-2]
	\arrow[""{name=0, anchor=center, inner sep=0}, "G", from=1-1, to=1-3]
	\arrow[""{name=1, anchor=center, inner sep=0}, "{T^G}"{description}, dashed, from=2-2, to=1-3]
	\arrow[""{name=2, anchor=center, inner sep=0}, "{1_{\mathcal{C}}}"', curve={height=18pt}, from=2-2, to=1-3]
	\arrow["\gamma"', shorten <=5pt, shorten >=6pt, Rightarrow, from=2-2, to=0]
	\arrow["\eta"', shorten <=4pt, shorten >=4pt, Rightarrow, from=2, to=1]
\end{tikzcd}\]

There is a natural transformation $ T^GT^G\circ G\xrightarrow{T^G\gamma} T^G\circ G\xrightarrow{\gamma}G$ and therefore by the universal property of right Kan extensions there exists a unique natural transformation $\mu:T^GT^G\to T^G$ such that the following diagrams are equal:

\[\begin{tikzcd}
	\mathcal{D} & {} & \mathcal{C} \\
	{} & \mathcal{C} & {}
	\arrow["G"', from=1-1, to=2-2]
	\arrow["{T^GT^G}"', from=2-2, to=1-3]
	\arrow[""{name=0, anchor=center, inner sep=0}, "G", from=1-1, to=1-3]
	\arrow[shorten <=5pt, shorten >=6pt, Rightarrow, from=2-2, to=0]
\end{tikzcd}
=
\begin{tikzcd}
	\mathcal{D} & {} &  \mathcal{C} \\
	{} & \mathcal{C} & {}
	\arrow["G"', from=1-1, to=2-2]
	\arrow[""{name=0, anchor=center, inner sep=0}, "G", from=1-1, to=1-3]
	\arrow[""{name=1, anchor=center, inner sep=0}, "{T^G}"{description}, dashed, from=2-2, to=1-3]
	\arrow[""{name=2, anchor=center, inner sep=0}, "{T^GT^G}"', curve={height=18pt}, from=2-2, to=1-3]
	\arrow["\mu"', shorten <=4pt, shorten >=4pt, Rightarrow, from=2, to=1]
	\arrow["\gamma"', shorten <=5pt, shorten >=6pt, Rightarrow, from=2-2, to=0]
\end{tikzcd}\]

\begin{prop}
The triple $(T^G,\eta,\mu)$ is a monad.\label{codensity}
\end{prop}

A proof for this result can be found in section 2 of \cite{leinster}. The monad in Proposition \ref{codensity} is called the \textbf{codensity monad of $G$}.

\begin{eg}
The codensity monad of a right adjoint functor $G:\mathcal{D}\to \mathcal{C}$ is the monad induced by the adjunction. This follows from the fact that $\text{Ran}_GG=GF$, where $F$ is the left adjoint of $G$.
\end{eg}

\begin{eg}[Kennison and Gildenhuys]
Let $\textbf{Set}_f$ be the category of finite sets and maps. The codensity monad of the inclusion functor $\textbf{Set}_f\to \textbf{Set}$ is the ultrafilter monad. This is a result from \cite{Kennison}. 
\end{eg}
The following result, which follows from Theorem 3.7.2 in \cite{borceuxI},  will be useful to find codensity monads of certain functors.
\begin{prop}\label{Kan}
Let $G:\mathcal{D}\to \mathcal{C}$ be a functor, where $\mathcal{D}$ is an essentially small category and $\mathcal{C}$ is a complete category. Then the codensity monad of $G$ exists and \[T^G(X)=\lim(X\downarrow G\xrightarrow{U}\mathcal{D}\xrightarrow{G}\mathcal{C})\]
for every object $X$ in $\mathcal{C}$. Here $U:X\downarrow G\to \mathcal{D}$ is the forgetful functor.
\end{prop}

\section{Premeasurable spaces}\label{premeasurablespaces}
In this section we will discuss probability premeasures and their finitely additive analogues on premeasurable spaces. We will give an overview of results on integration with respect to a probability premeasure and end the section with an integral representation theorem. In the case of probability measures, these correspond to standard results in measure theory. Premeasures play an important role in extension theorems such as the Carath\'eodory extension theorem. Because integration with respect to a premeasure is far less common than the usual Lebesgue integral, we will give detailed proofs for all results.

We will call a set $X$ together with an algebra of subsets $\mathcal{B}_X$ (i.e. a family of subsets of $X$ that is closed under complements and finite intersections) a \textbf{premeasurable space}. We say that a map $f:X\to Y$ between premeasurable spaces is \textbf{premeasurable} if $f^{-1}(\mathcal{B}_Y)\subseteq \mathcal{B}_X$.

Let $\mathbb{P}:\mathcal{B}_X\to [0,1]$ be a function such that $\mathbb{P}(X)=1$. The function $\mathbb{P}$ is called a \textbf{probability premeasure} if $\mathbb{P}\left(\bigcup_{n=1}^\infty A_n\right)=\sum_{n=1}^\infty \mathbb{P}(A_n)$ for every pairwise disjoint collection $(A_n)_{n=1}^\infty$ of subsets in $\mathcal{B}_X$ such that $\bigcup_{n=1}^\infty A_n$ is also in $\mathcal{B}_X$. We say that $\mathbb{P}$ is a \textbf{probability measure} if $\mathcal{B}_X$ is a $\sigma$-algebra.
If we only have that $\mathcal{P}(A\cup B) = \mathbb{P}(A)+\mathbb{P}(B)$ for disjoint subsets $A$ and $B$ in $\mathcal{B}_X$, then we call $\mathbb{P}$ a \textbf{finitely additive probability premeasure}. This is also known as a \textbf{charge}. If $\mathcal{B}_X$ is a $\sigma$-algebra we say that $\mathbb{P}$ is a \textbf{finitely additive probability measure}. Note that every probability (pre)measure is a finitely additive probability (pre)measure.
\begin{eg}
The set $[0,1]$ together with the smallest algebra that contains all the open intervals of $[0,1]$ is a premeasurable space.
\end{eg}
Let $(X,\mathcal{B}_X)$ be a premeasurable space. We will call a function $s:X\to [0,1]$ a \textbf{simple function on $X$} if there exists a natural number $n\geq 1$ and $a_1, a_2,\ldots, a_n \in [0,1]$ and $A_1,A_2,\ldots,A_n\in \mathcal{B}_X$ such that $$s=\sum_{k=1}^na_k1_{A_k}.$$ Note that a simple function is always premeasurable. The collection of simple functions on $X$ is denoted by $\text{Simp}(X,[0,1])$. Given a finitely additive probability premeasure $\mathbb{P}$, we can define a map $J_{\mathbb{P}}:\text{Simp}(X,[0,1])\to [0,1]$ by the assignment $$\sum_{k=1}^na_k1_{A_k} \mapsto \sum_{k=1}^n a_k \mathbb{P}(A_k).$$
Because $\mathcal{B}_X$ is closed under complements and finite unions, we can use a common refinement argument to show that the assignment is independent of the representation of the simple function. 

Again using a common refinement argument, we obtain the following proposition.
\begin{prop} \label{simple}
For simple functions $s,t:X\to [0,1]$ such that $s+t\leq 1$ and $r\in [0,1]$, we have that also $s+t$ and $rs$ are simple functions. Furthermore, $J_{\mathbb{P}}(s+t)=J_{\mathbb{P}}(s)+J_{\mathbb{P}}(t)$ and $J_{\mathbb{P}}(rs)=rJ_{\mathbb{P}}(s)$. If $s\leq t$, then also $t-s$ is a simple function and $J_{\mathbb{P}}(s)\leq J_{\mathbb{P}}(t)$.
\end{prop}

The following useful lemma follows from the construction described in the proof of Corollary 4.5.9 in \cite{Rao}.
\begin{lem}
Let $f:X\to [0,1]$ be a premeasurable map. There exists an increasing sequence $(s_n)_n$ of simple functions that converges uniformly to $f$ and there exists a decreasing sequence $(t_n)_n$ of simple functions that converges uniformly to $f$. \label{lemma}
\end{lem}

Let $\textbf{PreMble}(X,[0,1])$ be the set of premeasurable maps from $X$ to $[0,1]$. We define a map $I_{\mathbb{P}}:\textbf{PreMble}(X,[0,1])\to [0,1]$ by the assignment $$ f\mapsto \text{sup}\left\{J_{\mathbb{P}}(s)\mid s\leq f \text{ and }s\in \text{Simp}(X,[0,1])\right\}.$$
This map is well-defined because $J_{\mathbb{P}}$ is order preserving by Proposition \ref{simple}.

The following proposition summarizes results about the additivity and continuity of $I_{\mathbb{P}}$. In the case of probability measures, these are classical results in measure theory. For finitely additive probability premeausres these are lesser-known. Similar results have been discussed in \cite{sipos}.

\begin{prop}\label{integral}
Let $X$ be a premeasurable space and let $\mathbb{P}$ be a finitely additive probability premeasure on $X$. We have the following properties: 
\begin{enumerate}[label=(\roman*)]
    \item For a simple function $s:X\to [0,1]$, we have  $I_{\mathbb{P}}(s)=J_{\mathbb{P}}(s)$. In particular $I_{\mathbb{P}}(1)=1$. 
    \item \label{order} For premeasurable functions $f,g:X\to [0,1]$ such that $f\leq g$, we have $I_{\mathbb{P}}(f)\leq I_{\mathbb{P}}(g)$. 
    \item \label{inf} For a premeasurable map $f:X\to [0,1] $, \[I_{\mathbb{P}}(f)=\inf\left\{J_{\mathbb{P}}(s)\mid s\geq f \text{ and }s\in \textup{Simp}(X,[0,1])\right\}.\]
    \item For premeasurable maps $f,g:X\to [0,1]$ such that also $f+g\in\textbf{\textup{PreMble}}(X,[0,1])$, we have $I_{\mathbb{P}}(f+g)=I_{\mathbb{P}}(f)+I_{\mathbb{P}}(g)$.\label{fsum}
    \item Suppose that $\mathbb{P}$ is a probability premeasure. For an increasing sequence of premeasurable maps $(f_n:X\to [0,1])_{n=1}^\infty$ such that $f:=\lim_{n\to\infty}f_n$ is also premeasurable, $\lim_{n\to \infty}I_{\mathbb{P}}(f_n)=I_{\mathbb{P}}(f)$.\label{monotone}
    \item \label{series} Let $(f_n:X\to [0,1])_{n=1}^\infty$ be a collection of premeasurable maps such that $f:=\sum_{n=1}^\infty f_n$ is also an element of $\textbf{\textup{PreMble}}(X,[0,1])$. If $\mathbb{P}$ is a probability premeasure, then $$I_{\mathbb{P}}(f)=\sum_{n=1}^\infty I_{\mathbb{P}}(f_n).$$ 
\end{enumerate} 
\end{prop}
\begin{proof}
\begin{enumerate}[label=(\roman*)]
    \item Since $s\leq s$ it follows by the definition of $I$ that $J_{\mathbb{P}}(s)\leq I_{\mathbb{P}}(s)$. For every simple function $t$ such that $t\leq s$ we have by Proposition \ref{simple} that $J_{\mathbb{P}}(t)\leq J_{\mathbb{P}}(s)$ and therefore $I_{\mathbb{P}}(s)\leq J_{\mathbb{P}}(s)$. 
    \item This follows from the fact that for a simple function $s$ such that $s\leq f$ we also have $s\leq g$. 
    \item The ‘$\leq$’ inequality is clear. Now consider an $\epsilon \in (0,1]$ and simple functions $s$ and $t$ such that $s\leq f \leq t$ and such that $\lVert t-s\rVert_{\infty}\leq \epsilon$, which exist by Lemma \ref{lemma}. We find the following inequalities:
    
    $$J_{\mathbb{P}}(t)=J_{\mathbb{P}}(t-s+s) = J_{\mathbb{P}}(t-s)+J_{\mathbb{P}}(s)\leq J_{\mathbb{P}}(\epsilon) + J_{\mathbb{P}}(s) \leq \epsilon + I_{\mathbb{P}}(f).$$
    Here we used Proposition \ref{simple} and the definition of $I_{\mathbb{P}}$. The other inequality now follows. 
    \item Let $\epsilon>0$. By the definition of $J_{\mathbb{P}}$ and \ref{inf}, there exist simple functions $s_f,t_f,s_g$ and $t_g$ such that $s_f\leq f \leq t_f$ and $s_g\leq g \leq t_g$ and such that $$J_{\mathbb{P}}(t_f)-\epsilon \leq I_{\mathbb{P}}(f)\leq J_{\mathbb{P}}(s_f)+\epsilon$$ and $$J_{\mathbb{P}}(t_g)-\epsilon \leq I_{\mathbb{P}}(g)\leq J_{\mathbb{P}}(t_f)+\epsilon.$$
    
    Since $s_f+s_g$ and $(t_f+t_g) \wedge 1$ are simple functions by Proposition \ref{simple} and $s_f+s_g\leq f+g \leq (t_f+t_g)\wedge 1$, we find that $$I_{\mathbb{P}}(f+g)-2\epsilon \leq J_{\mathbb{P}}((t_f+t_g)\wedge 1) - 2\epsilon \leq J_{\mathbb{P}}(t_f)+J_{\mathbb{P}}(t_g)-2\epsilon \leq I_{\mathbb{P}}(f)+I_{\mathbb{P}}(g)$$ and $$I_{\mathbb{P}}(f)+I_{\mathbb{P}}(g)\leq J_{\mathbb{P}}(s_f)+J_{\mathbb{P}}(s_g) + 2\epsilon = J_{\mathbb{P}}(s_f+s_g)+2\epsilon \leq I_{\mathbb{P}}(f+g)+2\epsilon.$$
    Here we again used Proposition \ref{simple} and \ref{inf} and the definition of $I_{\mathbb{P}}$. The result follows by letting $\epsilon\to 0$.
    \item Since $f_n\leq f$ we have by \ref{order} that $I_{\mathbb{P}}(f_n)\leq I_{\mathbb{P}}(f)$ and therefore $\lim_{n\to \infty}I_{\mathbb{P}}(f_n)\leq I_{\mathbb{P}}(f)$. Now consider a simple function $s=\sum_{k=1}^ma_k1_{A_k}$ such that $s\leq f$. For $r\in [0,1)$ and a natural number $n$ define the set $$E_{n,r}:=\left\{x\in X\mid f_n(x)\geq rs(x)\right\}.$$
    
    Note that $E_{n,r}=\bigcup_{k=1}^m (f_n^{-1}([ra_k,1])\cap A_{k})$ and therefore it is a subset in $\mathcal{B}_X$.
    
    The function $rs1_{E_{n,r}}$ is simple and satisfies $rs1_{E_{n,r}}\leq f_n$. It follows now that $$I_{\mathbb{P}}(f_n)\geq J_{\mathbb{P}}(rs1_{E_{n,r}})=\sum_{k=1}^mra_k\mathbb{P}(A_k\cap E_{n,r}).$$
    
    Taking the limit $n\to \infty$ on both sides of the inequality gives us that \begin{equation}\lim_{n}I_{\mathbb{P}}(f_n)\geq \sum_{k=1}^mra_k \left(\lim_{n\to \infty}\mathbb{P}(A_k\cap E_{n,r})\right) \label{1}.\end{equation}
    
    It is easy to verify that $(E_{n,r})_{n\in \mathbb{N}}$ increases to $X$. Now define $F_{1,r}:=E_{1,r}$ and $F_{n,r}:=E_{n,r}\setminus F_{n-1,r}$ and note that $(F_{n,r}\cap A_k)_{n\in \mathbb{N}}$ is a collection of pairwise disjoint subsets in $\mathcal{B}_X$ such that their union is $A_k$. Using that $\mathbb{P}$ is a probability premeasure we find that \begin{equation}
     \mathbb{P}(A_k)=\sum_{n=1}^\infty\mathbb{P}(F_{n,r}\cap A_k)=\lim_{n\to\infty}\mathbb{P}\left(\bigcup_{l=1}^nF_{l,r}\cap A_k\right) = \lim_{n\to \infty}\mathbb{P}(E_{n,r}\cap A_k)\label{2}
    \end{equation}
    Combining \eqref{1} and \eqref{2} gives us $\lim_{n\to \infty}I_{\mathbb{P}}(f_n)\geq rJ(s)$ for all $r\in [0,1)$, which implies that $$\lim_{n\to \infty}I_{\mathbb{P}}(f_n)\geq J_{\mathbb{P}}(s).$$ 
    Since this holds for any simple function $s$ such that $s\leq f$, we can conclude that $\lim_{n\to \infty}I_{\mathbb{P}}(f_n)\geq I_{\mathbb{P}}(f)$.
    \item We first show the result for the case that every $f_n$ is simple. Because the finite sum of simple functions is again a simple function we can use \ref{fsum} and \ref{monotone} to show the following equalities \footnote{Note that the finite sum of premeasurable maps $X\to [0,1]$ is not necessarily premeasurable. Consider for example the set $[0,1]^2$ endowed with the Boolean algebra generated by open rectangles. The two projection maps $\pi_1,\pi_2:[0,1]^2\to [0,1]$ are premeasurable, but their sum is not.}:
    
    \[I_{\mathbb{P}}(f)=I_{\mathbb{P}}\left(\lim_{n\to \infty} \sum_{k=1}^nf_k\right)=\lim_{n\to \infty}I_{\mathbb{P}}\left(\sum_{k=1}^nf_k\right)=\lim_{n\to \infty}\sum_{k=1}^nI_{\mathbb{P}}(f_k)=\sum_{n=1}^\infty I_{\mathbb{P}}(f_k).\]
    
    For the general case we use Lemma \ref{lemma} to write $f_n$ as $\sum_{k=1}^\infty s_{k,n}$, where $s_{k,n}$ is a simple function for every $k$ and $n$. This gives us that $$f=\sum_{n=1}^\infty \sum_{k=1}^\infty s_{k,n}.$$
    
    Let $\psi:\mathbb{N}\setminus\{0\}\to \mathbb{N}\setminus\{0\}\times\mathbb{N}\setminus\{0\}$ be a bijection. We can now rewrite $f$ as follows:
    
    $$f=\sum_{m=1}^\infty s_{\psi(m)}.$$
    
    By the above it follows now that $$I_{\mathbb{P}}(f)=\sum_{m=1}^\infty I_{\mathbb{P}}(s_{\psi(m)}) = \sum_{n=1}^\infty\sum_{k=1}^\infty I_{\mathbb{P}}(s_{k,n}) = \sum_{n=1}^\infty I_{\mathbb{P}}(f_n).$$
\end{enumerate}
\end{proof}
\begin{rem}
In the case that $\mathcal{B}_X$ is a $\sigma$-algebra, $I_{\mathbb{P}}$ becomes the usual integration operation and Proposition \ref{integral}\ref{monotone} and \ref{series} become the usual monotone convergence theorems. Therefore we will also for the other cases write $\int_Xf\text{d}\mathbb{P}$ or $\int_X f(x)\mathbb{P}(\text{d}x)$ for $I_{\mathbb{P}}(f)$.
\end{rem}
\begin{prop}\label{representation0}
Let $I:\textbf{\textup{PreMble}}(X,[0,1])\to [0,1]$ be a map such that $I(1)=1$. If $I(f+g)=I(f)+I(g)$ for all premeasurable maps $f,g:X\to [0,1]$ such that also $f+g\in \textbf{\textup{PreMble}}(X,[0,1])$, then there exists a unique finitely additive probability premeasure $\mathbb{P}$ such that \[I=I_{\mathbb{P}}.\]
\end{prop}

\begin{proof}
Define $\mathbb{P}(A):=I(1_A)$ for all $A$ in $\Sigma_X$. Clearly, $\mathbb{P}$ is a finitely additive probability premeasure. 

We have $I(qf)=qI(f)$ for all $q\in [0,1]\cap \mathbb{Q}$ and for all premeasurable functions $f:X\to [0,1]$. For a simple function $s$ such that $s\leq f$ we have that $f-s$ is a premeasurable function and it follows that $I(s)\leq I(f)$.

For a premeasurable function $f:X\to [0,1]$ and for $\epsilon>0$ there exists a simple function $s=\sum_{k=1}^ma_k1_{A_k} \leq f$ such that $I_{\mathbb{P}}(f)\leq \sum_{k=1}^m a_k \mathbb{P}(A_k) +\epsilon.$
We can assume that $a_k$ is an element of $[0,1]\cap \mathbb{Q}$ for every $k$. We now see that 

$$I_\mathbb{P}(f)\leq \sum_{k=1}^ma_kI(1_{A_k})+\epsilon =I(s)+\epsilon \leq I(f)+\epsilon.$$
Letting $\epsilon\to 0$ gives us that $I_{\mathbb{P}}(f)\leq I(f)$. Using Proposition \ref{integral}\ref{inf} we obtain the other inequality in a similar way.

Let $\mathbb{P}'$ be another finitely additive probability premeasure with this property; then $$\mathbb{P}'(A)=I_{\mathbb{P}'}(1_A)=I(1_A)=I_\mathbb{P}(1_A)=\mathbb{P}(A)$$ for every $A$ in $\mathcal{B}_X$. This implies $\mathbb{P}'=\mathbb{P}$.
\end{proof}
\begin{prop}\label{representation}
Let $I:\textbf{\textup{PreMble}}(X,[0,1])\to [0,1]$ be a map such that $I(1)=1$. If $I\left(\sum_{n=1}^\infty f_n\right )=\sum_{n=1}^\infty I(f_n)$ for every collection of premeasurable maps $(f_n:X\to [0,1])_{n=1}^\infty$ such that $f:=\sum_{n\in A} f_n$ is also an element of $\textbf{\textup{PreMble}}(X,[0,1])$ for every finite or cofinite subset $A$ of $\mathbb{N}\setminus\{0\}$, then there exists a unique probability premeasure $\mathbb{P}$ such that \[I=I_{\mathbb{P}}.\]
\end{prop}
\begin{proof}
By Proposition \ref{representation0} there exists a unique finitely additive probability measure $\mathbb{P}$ such that $I=I_{\mathbb{P}}$. For a pairwise disjoint collection $(A_n)_{n=1}^\infty$ in $\mathcal{B}_X$ such that $\bigcup_{n=1}^\infty A_n$ is also an element of $\mathcal{B}_X$, we see that $\sum_{n\in A}1_{A_n}$ is premeasurable for every finite or cofinite subset $A$ of $\mathbb{N}\setminus\{0\}$. It follows now that $$\mathbb{P}\left(\bigcup_{n=1}^\infty A_n\right) = I_\mathbb{P}\left(\sum_{n=1}^\infty1_{A_n}\right)=\sum_{n=1}^\infty I_{\mathbb{P}}(A_n)=\sum_{n=1}^\infty \mathbb{P}(A_n)$$
\end{proof}
\begin{cor}\label{repcor}
Let $I:\textbf{Mble}(X,[0,1])\to [0,1]$ be a map such that $I(1)=1$. If $I\left(\sum_{n=1}^\infty f_n\right)=\sum_{n=1}^\infty I(f_n)$ for every collection of measurable maps $(f_n:X\to [0,1])_{n=1}^\infty$ such that $\sum_{n=1}^\infty f_n \leq 1$, then there exists a unique probability measure $\mathbb{P}$ such that $$I=I_\mathbb{P}.$$
\end{cor} \section{Probability monads on categories of premeasurable spaces}\label{probabilitymonadsoncategoriesofpremeasurablespaces}
In this section we will present several monads of (finitely additive) probability (pre)measures on categories of (pre)measurable spaces. We explain how each of these monads can be constructed as a codensity monad.

We will denote the category of measurable spaces and measurable maps  by $\textbf{Mble}$ and the category of premeausrable spaces and premeasurable maps by $\textbf{PreMble}$.  We write $\textbf{Set}_c$ for the category of countable sets and functions. 

\subsection{Giry monad of probability measures}\label{4.1}
Here we will discuss the monad of probability measures on the category of measurable spaces, which is known as the Giry monad. We show how this monad arises as the codensity monad of a functor $G:\textbf{Set}_c\to \textbf{Mble}$.

Let $(X,\Sigma_X)$ be a measurable space and let $\mathcal{G}X$ be the set of all probability measures on $X$. For a measurable subset $A$ of $X$ let $\text{ev}_A:\mathcal{G}X\to [0,1]$ denote the function defined by the assignment $\mathbb{P}\mapsto \mathbb{P}(A)$. The set $\mathcal{G}X$ becomes a measurable space by endowing it with the smallest $\sigma$-algebra that makes $\text{ev}_A$ measurable for all $A$ in $\Sigma_X$. We will denote this measurable space also by $\mathcal{G}X$.

Let $f:X\to Y$ be a measurable map between measurable spaces $X$ and $Y$. Every probability measure $\mathbb{P}$ induces a probability measure $\mathbb{P}\circ f^{-1}$ on $Y$ which is defined by $$\mathbb{P}\circ f^{-1}(B):=\mathbb{P}(f^{-1}(B))$$ for all $B$ in $\Sigma_Y$. This probability measure is called the \textbf{pushforward measure}. The assignment $\mathbb{P}\mapsto \mathbb{P}\circ f^{-1}$ defines a map $\mathcal{G}X\to \mathcal{G}Y$, which we will denote by $\mathcal{G}f$. It can be checked that $\mathcal{G}f$ is measurable. 

Let $\textbf{Mble}$ be the category of measurable spaces and measurable maps. The assignments $X\mapsto \mathcal{G}X$ and $f\mapsto Gf$ define a functor $\mathcal{G}:\textbf{Mble}\to \textbf{Mble}$.

For every measurable space $X$ we have a measurable map $\eta_X:X\to \mathcal{G}X$ that sends an element $x\in X$ to the probability measure $\delta_x$ that is defined by \begin{equation}\label{unit}
    \delta_x(A):=\begin{cases}1 \text{ if } x\in A \\
0 \text{ otherwise.}\end{cases}\end{equation}
Moreover, these maps form a natural transformation $\eta:1_{\textbf{Mble}}\to \mathcal{G}$.

We also have a measurable map $\mu_X:\mathcal{G}\mathcal{G}X\to \mathcal{G}X$ that sends a probability measure $\textbf{P}$ on $\mathcal{G}X$ to the probability measure $\mu_X(\textbf{P})$ on $X$ that is defined by \begin{equation}\label{multiplication}
    \mu_X(\textbf{P})(A):=\int_{\mathcal{G}X}\mathbb{P}(A)\textbf{P}(\text{d}\mathbb{P})\end{equation}
for all $A\in \Sigma_X$. Also these maps form a natural transformation $\mu:\mathcal{G}\mathcal{G}\to\mathcal{G}$. 
\begin{prop}[Giry, \cite{giry}]\label{Giry}
The triple $(\mathcal{G},\eta,\mu)$ is a monad. 
\end{prop}
The monad in Proposition \ref{Giry} is called the \textbf{Giry monad (of probability measures)}. 

Every countable set $A$ can be turned into a measurable spaces, namely the set $A$ endowed with the whole powerset of $A$ as $\sigma$-algebra. Every function of countable sets becomes measurable with respect to these $\sigma$-algebras. This leads to a functor $j:\textbf{Set}_c\to \textbf{Mble}$. Define the functor $G$ as $$\textbf{Set}_c\xrightarrow{j}\textbf{Mble}\xrightarrow{\mathcal{G}}\textbf{Mble}$$

This means that for a countable set $A$ the underlying set of $GA$ is equal to $$\{(p_a)_a\in [0,1]^A\mid \sum_{a\in A}p_a =1\}.$$ For a map $f:X\to GA$ we will use the notation $f_a$ to mean $\text{ev}_{\{a\}}\circ f$. Note that a map $f:X\to GA$ is measurable if and only if $f_a$ is measurable for every $a$. 

For a map of countable sets $f:A\to B$ the measurable map $Gf:GA\to GB$ is given by the assignment $$(p_a)_{a\in A}\mapsto \left(\sum_{a\in f^{-1}(b)}p_a\right)_{b\in B}.$$

\begin{thm}
The Giry monad is the codensity monad of $G$. \label{Girymonad}
\end{thm}

\begin{proof}
Proposition \ref{Kan} tells us that the codensity monad of $G$ exists.
We will now show that for all measurable spaces $X$, \[\mathcal{G}(X)=\lim(X\downarrow G\xrightarrow{U} \textbf{Set}_c\xrightarrow{G}\textbf{Mble}).\] 
Proposition \ref{Kan} then implies that $T^G(X)\cong \mathcal{G}(X)$ for all measurable spaces $X$.

For a measurable map $f:X\to GA$ define a map $p_f$ as $$\mathcal{G}X\xrightarrow{\mathcal{G}f}\mathcal{G}\mathcal{G}jA\xrightarrow{\mu_{jA}}\mathcal{G}jA=GA.$$ This means that for $\mathbb{P}\in \mathcal{G}X$, $$p_f(\mathbb{P})=\left(\int_X f_a\text{d}\mathbb{P}\right)_{a\in A}.$$
Consider a commutative triangle \[\begin{tikzcd}
	& X \\
	GA && GB.
	\arrow["f"', from=1-2, to=2-1]
	\arrow["Gs"', from=2-1, to=2-3]
	\arrow["g", from=1-2, to=2-3]
\end{tikzcd}\]
We have the following equalities: 
\begin{align*}
    Gs\circ p_f & =Gs\circ \mu_{jA}\circ \mathcal{G}f\\
    & = \mathcal{G}js\circ \mu_{jA}\circ \mathcal{G}f\\
    & = \mu_{jB}\circ \mathcal{G}\mathcal{G}js\circ \mathcal{G}f\\
    & = \mu_{jB}\circ \mathcal{G}g = p_g
\end{align*} 

In more measure theoretic terms this means that for a probability measure $\mathbb{P}$ on $X$ and an element $b$ in $B$ we have the following: 
\begin{align*}
(Gs\circ p_f)(\mathbb{P})_b & = \sum_{a\in s^{-1}(b)} p_f(\mathbb{P})_a\\
& = \sum_{a\in s^{-1}(b)} \int_X f_a\text{d}\mathbb{P}\\
& = \int_X \sum_{a\in s^{-1}(b)}f_a\text{d}\mathbb{P}\\
& = \int_X (Gs\circ f)_b\text{d}\mathbb{P}\\
& = \int_X g_b\text{d}\mathbb{P}=p_g(\mathbb{P})_b.
\end{align*}

We can conclude that $(\mathcal{G}X,(p_f)_f)$ forms a cone over the diagram $X\downarrow G \xrightarrow{U} \textbf{Set}_c\xrightarrow{G}\textbf{Mble}$. We will now show that this is the limiting cone of the diagram. To do this let us consider some cone $(Y,(q_f)_f)$ over the diagram. Let $\textbf{2}:=\{0,1\}$ and $\textbf{1}:=\{0\}$. For a measurable map $f:X\to [0,1]$ let $\widehat{f}:X\to G\textbf{2}$ be the measurable map that sends an element $x$ in $X$ to $(1-f(x),f(x))$.

For an element $y\in Y$ define a map $I_y:\textbf{Mble}(X,[0,1])\to [0,1]$ by $$I_y(f):=q_{\widehat{f}}(y)_1.$$

Let $t:\textbf{1}\to \textbf{2}$ be the map that sends $0$ to $1$ and let $e$ be the unique measurable map $X\to G\textbf{1}$. We find the following commutative triangle: \[\begin{tikzcd}
	{} & X \\
	{G\textbf{2}} && {G\textbf{1}}.
	\arrow["{\widehat{1}}"', from=1-2, to=2-1]
	\arrow["Gt", from=2-3, to=2-1]
	\arrow["e", from=1-2, to=2-3]
\end{tikzcd}\]
Because $(Y,(q_f)_f)$ is a cone over the diagram we also have that the following triangle commutes:
\[\begin{tikzcd}
	{} & Y \\
	{G\textbf{2}} && {G\textbf{1}}.
	\arrow["{q_{\widehat{1}}}"', from=1-2, to=2-1]
	\arrow["Gt", from=2-3, to=2-1]
	\arrow["{q_e}", from=1-2, to=2-3]
\end{tikzcd}\]
It now follows that $I_y(1)=q_{\widehat{1}}(y)_1=1$.

For a collection $(f_n:X\to [0,1])_{n=1}^\infty$ of measurable maps such that $f:=\sum_{n=1}^\infty f_n$ is also an element of $\textbf{Mble}(X,[0,1])$, let $h:X\to G\mathbb{N}$ be the measurable map defined by $h(x)_n:=f_n(x)$ for $n\geq 1$ and $h(x)_0:=1-f(x)$. For $n\geq 1$ let $s_n:\mathbb{N}\to \textbf{2}$ be the map that sends $n$ to $1$ and every other element to $0$. Let $s:\mathbb{N}\to \textbf{2}$ be the map that sends $0$ to $0$ and every other element to $1$. We have the following commutative diagrams: 

\[\begin{tikzcd}
	& X &&& X \\
	{G\mathbb{N}} && {G\textbf{2}} & {G\mathbb{N}} && {G\textbf{2}}
	\arrow["h"', from=1-2, to=2-1]
	\arrow["{\widehat{f_n}}", from=1-2, to=2-3]
	\arrow["{Gs_n}"', from=2-1, to=2-3]
	\arrow["h"', from=1-5, to=2-4]
	\arrow["Gs"', from=2-4, to=2-6]
	\arrow["{\widehat{f_0}}", from=1-5, to=2-6]
\end{tikzcd}\]
Therefore also the following triangles commute: 
\[\begin{tikzcd}
	& Y &&& Y \\
	{G\mathbb{N}} && {G\textbf{2}} & {G\mathbb{N}} && {G\textbf{2}}
	\arrow["{q_h}"', from=1-2, to=2-1]
	\arrow["{q_{\widehat{f_n}}}", from=1-2, to=2-3]
	\arrow["{Gs_n}"', from=2-1, to=2-3]
	\arrow["h"', from=1-5, to=2-4]
	\arrow["Gs"', from=2-4, to=2-6]
	\arrow["{q_{\widehat{f}}}", from=1-5, to=2-6]
\end{tikzcd}\]
It follows now that
\begin{equation}
    I_y(f_0)=q_{\widehat{f}}(y)_1=(Gs\circ q_h(y))_1=\sum_{n\in \mathbb{N}} q_h(y)_n \label{3}
\end{equation}
and that for every $n\geq 1$,
\begin{equation}
    I_y(f_n)=q_{\widehat{f_n}}(y)_1=(Gs_n\circ q_h(y))_1= q_h(y)_n. \label{4}
\end{equation}
Combining \eqref{3} and \eqref{4} gives us that $I_y(f)=\sum_{n\in \mathbb{N}} I_y(f_n)$. By Corollary \ref{repcor} it follows that there exists a unique probability measure $\mathbb{P}_y$ such that $I_y=I_{\mathbb{P}_y}$. The assignment $y\mapsto \mathbb{P}_y$ defines a map $q:Y\to\mathcal{G}X$. We have that $\text{ev}_A\circ q=(q_{\widehat{1_A}})_1$ for all $A\in \Sigma_X$ and therefore $q$ is measurable. 

Let $f:X\to GA$ be a measurable map and let $a$ be an element of $A$. Let $s_a:A\to \textbf{2}$ be the map that sends $a$ to $1$ and every other element to $0$. Since we have that $Gs_a\circ f = \widehat{f_a}$ we also have that $Gs_a\circ q_f=q_{\widehat{f_a}}$. In particular we find for every $y\in Y$ that $$I_y(f_a)=q_{\widehat{f_a}}(y)_1=(Gs_a\circ q_f(y))_1=q_f(y)_a.$$
Using this we obtain for every $y\in Y$ and for every $a\in A$ that $$p_f\circ q(y)_a=I_{\mathbb{P}_y}(f_a)=I_y(f)_a=q_f(y)_a.$$
This shows that $q$ is a morphism of cones from $(\mathcal{G}X,(p_f)_f)$ to $(Y,(q_f)_f)$. 

Let $\Tilde{q}:Y\to \mathcal{G}X$ be another morphism of cones. Then for every measurable subset $A$ of $X$ and for every $y$ in $Y$ we have that $$\Tilde{q}(y)(A)=(p_{\widehat{1}_A}\circ \Tilde{q}(y))_1=q_{\widehat{1_A}}(y)_1=q(y)(A).$$
This shows that $\Tilde{q}=q$ and therefore $(\mathcal{G}X,(p_f)_f)$ is the limiting cone over the diagram. This implies that $\mathcal{G}(X)\cong T^G(X)$ for all measurable spaces $X$. Moreover this induces a natural isomorphism $\mathcal{G}\cong T^G$. 

It is straightforward to check that the unit and multiplication of the codensity monad of $G$ are equal to the unit and multiplication of the Giry monad.
\end{proof}
\begin{rem}
 While Theorem \ref{Girymonad} states that $\text{Ran}_GG=\mathcal{G}$, it is also  true that $\text{Ran}_jG=\mathcal{G}$. This construction immediately gives probability measures as set functions, without using an integral representation theorem. 
\end{rem}
\begin{rem}
A different construction for the Giry monad as a codensity monad is given in \cite{Avery}. Avery shows that the codensity monad of the inclusion of the category of powers of the unit interval and affine maps in $\textbf{Mble}$ is isomorphic to the Giry monad.  
\end{rem}

\subsection{Giry monad of finitely additive probability measures}
In the same way as the Giry monad of probability measures was defined, we can define a monad $(\mathcal{G}_f,\eta,\mu)$ of finitely additive probability measures on $\textbf{Mble}$. We call this monad the \textbf{Giry monad of finitely additive probability measures}.

Let $\textbf{Set}_f$ be the category of finite sets and maps and let $i:\textbf{Set}_f\to \textbf{Set}_c$ be the inclusion functor. Let $G:\textbf{Set}_c\to \textbf{Mble}$ be as in the previous subsection and define $G_f:=G\circ i$.

\begin{thm}
The Giry monad of finitely additive probability measures is the codensity monad of $G_f$.\label{finadd}
\end{thm}

The proof for Theorem \ref{finadd} is similar to the proof of Theorem \ref{Girymonad}. There are two places where the proof is slightly different.  First, instead of using a countable index set $\mathbb{N}$, it is now enough to use a finite index set. Second, for this proof we use the integral representation result Proposition \ref{representation0} instead of Corollary \ref{repcor}.

\subsection{Giry monad of probability premeasures}\label{4.3}
In this section we will discuss how the monad of probability premeasures on the category of premeasurable spaces arises as the codensity monad of a functor $G_p$. This functor $G_p$ is similar to the functor $G$ in section \ref{4.1}, however here the domain needs to be restricted to finite maps, because we are working with premeasurable maps.

Let $(X,\mathcal{B}_X)$ be a premeasurable space and let $\mathcal{G}_pX$ be the set of all probability premeasures on $X$. For $A\in \mathcal{B}_X$, let $\text{ev}_A:\mathcal{G}_p(X)\to [0,1]$ be the map that sends a probability premeasure $\mathbb{P}$ to $\mathbb{P}(A)$. The set $\mathcal{G}_pX$ becomes a premeasurable space by endowing it with the smallest algebra that makes $\text{ev}_A$ premeasurable for all $A\in \mathcal{B}_X$. We will denote this premeasurable space also by $\mathcal{G}_pX$.

Every premeasurable map $f:X\to Y$ induces a premeasurable map $\mathcal{G}_p:\mathcal{G}_pX\to \mathcal{G}_pY$ by pushing forward probability premeasures along $f$.

Let $\textbf{PreMble}$ be the category of premeasurable spaces and premeasurable maps. The above defines a functor $\mathcal{G}_f:\textbf{PreMble}\to \textbf{PreMble}$. 

We can define natural transformations $\eta_p:1_{\textbf{PreMble}}\to \mathcal{G}_p$ and $\mu_p:\mathcal{G}_p\mathcal{G}_p\to \mathcal{G}_p$ in a similar way as we defined the unit \eqref{unit} and the multiplication \eqref{multiplication} for the Giry monad of probability measures. 

Similarly as in Proposition \ref{Giry} this construction gives us a monad.
\begin{prop}
The triple $(\mathcal{G}_p,\eta_p,\mu_p)$ is a monad.\label{GiryPreMble}
\end{prop}

We call the monad in Proposition \ref{GiryPreMble} the \textbf{Giry monad of probability premeasures}.

For a countable set $A$ let $F_A$ be the set of all finite and cofinite subsets of $A$. A map $f:A\to B$ between countable sets is called a \textbf{finite map} if $f^{-1}(F_B)\subseteq F_A$. 
Note that the composition of finite maps is a finite map. Let $\textbf{Set}_c^f$ be the category of finite sets and finite maps. Every countable set $A$ can be turned into a premeasurable space, namely the set $A$ together with the algebra $F_A$. Every finite map between countable sets becomes premeasurable with respect to these algebras. We obtain a functor $j_p:\textbf{Set}_c^f\to \textbf{PreMble}$.

Now define the functor $G_p$ as $$\textbf{Set}_c^f\xrightarrow{j_p}\textbf{PreMble}\xrightarrow{\mathcal{G}_p}\textbf{PreMble}.$$

For a countable set $A$ the space $G_pA$ is the set $\{(p_a)_a\in [0,1]^A\mid \sum_{a\in A}p_a=1\}$ together with the smallest algebra that makes $\text{ev}_{A'}$ premeasurable for every finite or cofinite subset $A'$ of $A$. A map $f:X\to G_pA$ is premeasurable if and only if $\text{ev}_{A'}\circ f$ is premeasurable for every $A'\in F_A$.

For a finite map of countable sets $f:A\to B$ the premeasurable map $G_pf:G_pA\to G_pB$ is given by the assignment $$(p_a)_{a\in A}\mapsto \left(\sum_{a\in f^{-1}(b)} p_a\right)_{b\in B}$$ 

\begin{thm}
The Giry monad of probability premeasures is the codensity monad of $G$. \label{premeasmonad}
\end{thm}
Because all the maps between countable sets in the proof of Theorem \ref{Girymonad} were finite, we can use a similar argument to prove Theorem \ref{premeasmonad}. However now we need to use Proposition \ref{representation} instead of its corollary (Corollary \ref{repcor}).

\begin{rem}
Let $i:\textbf{Set}_c^f\to \textbf{Set}_c$ be the inclusion functor. We have that the codensity monad of $Gi$ is the Giry monad of probability measures, since in the proof of Theorem \ref{Girymonad} we only used finite maps.
\end{rem}
\section{Hausdorff spaces}\label{Hausdorffspaces}
In section \ref{premeasurablespaces} we presented several results about probability measures on measurable spaces. If we assume that the measurable space has more structure, we can say more about the probability measures on that space. This will be the topic of the current section. We will assume extra topological structure on the space and study the consequences for the probability measures on that space. We end the section by a generalized Daniell-Stone theorem.

\subsection{Probability measures on Hausdorff spaces}
Let $X$ be a Hausdorff space. The smallest $\sigma$-algebra that contains all the open sets of $X$ is called the \textbf{Borel $\sigma$-algebra} and is denoted by $\text{Bo}_X$. The smallest $\sigma$-algebra that makes all the continuous functions $f:X\to [0,1]$ measurable is called the \textbf{Baire $\sigma$-algbera} and is denoted by $\text{Ba}_X$. The following result, which is Theorem 7.1.1 in \cite{dudley}, states that for metric spaces these $\sigma$-algebras are the same. 

\begin{prop}
For every Hausdorff space $X$ we have that $\text{Ba}_X\subseteq \text{Bo}_X$. For a metric space $X$ we have that $\text{Ba}_X = \text{Bo}_X$.
\end{prop}

A probability measure on $(X,\text{Bo}_X)$ is called a \textbf{Borel probability measure} and a probability measure on $(X,\text{Ba}_X)$ a \textbf{Baire probability measure}. A Borel probability measure $\mathbb{P}$ is called a \textbf{Radon probability measure} if 
$$\mathbb{P}(A)=\sup\{\mathbb{P}(K)\mid K\subseteq A \text{ and }K\text{ is compact}\}$$
for all $A$ in $\text{Bo}_X$. The next result tells us that on compact Hausdorff spaces, Baire probability measures correspond to Radon probability measures.

\begin{prop}
Every Baire probability measure on a compact Hausdorff space can be extended uniquely to a Radon probability measure. \label{Baire extension}
\end{prop}

A proof for Proposition \ref{Baire extension} can be found in \cite{dudley} (Theorem 7.1.5).

\subsection{Representations of probability measures on Hausdorff spaces}\label{5}
For a collection $L$ of real-valued functions on a set $X$ let $\sigma(L)$ denote the smallest $\sigma$-algebra on $X$ such that every $f\in L$ is measurable. Note that this $\sigma$-algebra is generated by sets of the form $\{f>r\}:=\{x\in X\mid f(x)>r\}$ for $f\in L$ and $r\in \mathbb{R}$.

\begin{defn}Let $X$ be a set and let $L\subseteq [0,\infty)^X$ be a subset of non-negative valued functions on $X$ together with the pointwise ordering. Let $\mathbb{N}L:=\{nf\mid f\in L\}$. We call $L$ a \textbf{weak integration lattice}\footnote{We use the term \emph{weak} to indicate that this is the weaker version of the integration lattices discussed in the Appendix. For example, the subset of all $1$-Lipschitz functions on a metric space form a \emph{weak} integration lattice, but not an integration lattice.} if:
\begin{itemize}
    \item $1\in L$,
    \item for all $f,g\in L$ we have  $f\vee g$, $f\wedge g$, $f\vee g-f\wedge g \in \mathbb{N}L$,
    \item for all $f\in L$ and for all $n\in \mathbb{N}$,  $nf\wedge 1 \in \mathbb{N}L$ and
    \item for all $f\in L$ and for all $r \in [0,1]$, $rf\in L$.
\end{itemize}
\end{defn}
\begin{defn}We call a map $I:L\to [0,\infty)$ a \textbf{weak integration operator} if $I(1)=1$ and for every collection $(f_n)_{n\in\mathbb{N}}$ in $L$ such that $f:=\sum_{n\in \mathbb{N}} f_n\in L$, we have $I(f)=\sum_{n\in \mathbb{N}} I(f_n)$.
\end{defn}
The following result is the Daniell-Stone representation theorem with weaker conditions. A proof, based on \cite{kindler}, is given in the Appendix.

\begin{thm}
Let $I$ be a weak integration operator on a weak integration lattice $L$ on a set $X$. There exists a unique probability measure $\mathbb{P}$ on $(X,\sigma(L))$ such that $$I_\mathbb{P}(f)=I(f)$$ for all $f\in L$. \label{representation2}
\end{thm}

\begin{prop}\label{Dini}
Let $X$ be a compact Hausdorff space and let $L$ be a weak integration lattice such that every $f\in L$ is continuous and such that $f+g\in \mathbb{N}L$ for all $f,g\in L$. Let $I:L\to [0,\infty)$ be a function such that $I(f+g)=I(f)+I(g)$ for $f,g\in L$ with $f+g\in L$. Suppose also that $I(1)=1$. Then there exists a unique probability measure $\mathbb{P}$ on $(X,\sigma(L))$ such that $I_\mathbb{P}(f)=I(f)$ for $f\in L$.
\end{prop}

The proof of Proposition \ref{Dini} of this is given in the appendix. The proof relies on Dini's theorem.

The above integral representation theorem can be used to prove several integral representation theorems of probability measures on certain topological spaces. 

\begin{eg}
Let $X$ be a metric space and let $L$ be the set of all $1$-Lipschitz functions on $X$ taking values in $[0,1]$. Then $L$ is a weak integration lattice and $\sigma(L)=\text{Ba}_X=\text{Bo}_X$. To see this, observe that for every closed set $A$, the function $f_A:=d(\cdot,A)\wedge 1:X\to [0,1]$ is $1$-Lipschitz and $A=f_A^{-1}(0)\in \sigma(L)$. Therefore $\text{Bo}_X\subseteq \sigma(L)$ and clearly $\sigma(L)\subseteq \text{Ba}_X$. 

It follows now by Theorem \ref{representation2} that every weak integration operator $I$ on $L$ induces a unique Borel probability measure on $X$ such that $I_\mathbb{P}=I$. \label{met1}
\end{eg}

\begin{eg}
Let $(X,d)$ be a metric space. A Borel probability measure $\mathbb{P}$ on $X$ is said to \textbf{have finite moment} if for all $x\in X$ we have $$\int_X d(x,y)\mathbb{P}(\text{d}y)<\infty.$$

Let $X$ be a metric space and let $L$ be the set of all $1$-Lipschitz functions on $X$ taking values in $[0,\infty)$. Then $L$ is a weak integration lattice and $\sigma(L)=\text{Ba}_X=\text{Bo}_X$. Theorem \ref{representation2} now tells us that there exists a unique Borel probability measure $\mathbb{P}$ such that $I_\mathbb{P}=I$. Since for every $x\in X$ the function $d(x,\cdot):X\to [0,\infty)$ is a $1$-Lipschitz function, we find 
$$\int_Xd(x,y)\mathbb{P}(\text{d}y) = I_\mathbb{P}(d(x,\cdot))=I(d(x,\cdot))<\infty$$
and therefore $\mathbb{P}$ has finite moment. \label{met1fm}
\end{eg}

\begin{eg}
In Example \ref{met1} and \ref{met1fm} we can also use all Lipschitz maps instead of only the $1$-Lipschitz maps.
\end{eg}

\begin{eg}[Riesz-Markov representation theorem]
Let $X$ be a compact Hausdorff space and let $L$ be the set of all continuous functions on $X$ taking values in $[0,1]$. Then $L$ is a weak integration lattice and $\sigma(L)=\text{Ba}_X$. 

Let $I:L\to [0,1]$ be a function such that $I(1)=1$ and such that $I$ preserves binary sums that exist in $L$. Then Proposition \ref{Dini} gives us a unique Baire probability measure $\mathbb{P}$ on $X$ such that $I_\mathbb{P}(f)=I(f)$ for all $f\in L$. Note that by Proposition \ref{Baire extension} the Baire probability measure $\mathbb{P}$ can be uniquely extended to a Radon probability measure. \label{ch}
\end{eg}

\begin{eg}
Let $X$ be a compact metric space and let $L$ be the set of all $1$-Lipschitz functions on $X$ taking values in $[0,1]$. Then $L$ is a weak integration lattice that satisfies the conditions of Proposition \ref{Dini}.

Therefore for a function $I:L\to [0,1]$ such that $I(f+g)=I(f)+I(g)$ for all $f,g\in L$ with $f+g$ also in $L$ and such that $I(1)=1$, there exists a unique Radon probability measure $\mathbb{P}$ on $X$ such that $I_{\mathbb{P}}=I.$ \label{compactmetric}
\end{eg}

\begin{eg}\label{eghaus}
Let $X$ be a Hausdorff space and let $L$ be the set of continuous functions $X\to[0,1]$. We see that $L$ is a weak integration lattice and $\sigma(L)=\text{Ba}_X$. Using Theorem \ref{representation2} we see that for every weak integration operator $I$ on $L$ there is a unique Baire probability measure $\mathbb{P}$ such that $I_\mathbb{P}=I$.
\end{eg}

\section{Probability monads on categories of Hausdorff spaces}\label{probabilitymonadsoncategoriesofHausdorffspaces}
In this section we will present monads of Radon probability measures and Baire probability measures and explain how they arise as codensity monads. We will introduce two new probability monads: the bounded Lipschitz monad and the Baire monad. 

We write $\textbf{Haus}$ for the category of Hausdorff spaces and continuous maps. The full subcategory of compact Hausdorff spaces is denoted by $\textbf{CH}$. The category of compact metric spaces and $1$-Lipschitz maps is denoted by $\textbf{KMet}_1$. For the category of finite sets and functions we write $\textbf{Set}_f$.

\subsection{Radon monad}\label{6.1}
We will discuss a monad of Radon probability measures on the category of compact Hausdorff spaces, known as the Radon monad. This monad was first introduced in \cite{semadeni} and \cite{swirszcz} and further discussed in \cite{jacobs} and \cite{keimel}.  We explain how this monad can be constructed as the codensity monad of a functor $R:\textbf{Set}_f\to \textbf{CH}$. Although Radon probability measures are $\sigma$\emph{-additive}, the domain of $R$ is the category of \emph{finite} sets and maps. Note that compact Hausdorff spaces arise as the algebras of the codensity monad of the inclusion $\textbf{Set}_f\to \textbf{Set}$ (i.e. the ultrafilter monad) as discussed in \cite{Kennison} and \cite{leinster}. 

Let $X$ be a topological space and let $\mathcal{R}X$ be the set of all Radon probability measures on $X$. Endow $\mathcal{R}X$ with the smallest topology such that the evaluation map $\text{ev}_f:\mathcal{R}X\to [0,1]$, that sends $\mathbb{P}$ to $\int_Xf\text{d}\mathbb{P}$, becomes continuous for every continuous function $f:X\to [0,1]$. This is the topology of weak convergence of probability measures. We will denote the obtained topological space also by $\mathcal{R}X$.

The following result follows from the Banach-Alaoglu theorem (section 3.15 in \cite{rudin}).
\begin{prop}
The topological space $\mathcal{R}X$ is a compact Hausdorff space.\label{RX}
\end{prop}

\begin{lem}
Let $f:X\to Y$ be a continuous function between compact Hausdorff spaces and let $\mathbb{P}$ be a Radon probability measure on $X$. Then also $\mathbb{P}\circ f^{-1}$ is a Radon probability measure. Furtermore, the assignment $\mathbb{P}\mapsto \mathbb{P}\circ f^{-1}$ defines a continuous function $\mathcal{R}f:\mathcal{R}X\to \mathcal{R}Y$.\label{Rf}
\end{lem}

The proof of Lemma \ref{Rf} is elementary.

Using Proposition \ref{RX} and Lemma \ref{Rf} we can define a functor $\mathcal{R}:\textbf{CH}\to \textbf{CH}$. 

For a compact Hausdorff $X$ we can define maps $\eta_X:X\to \mathcal{R}X$ and $\mu_X:\mathcal{R}\mathcal{R}X\to \mathcal{R}X$ in the same way as we did for the Giry monad. These maps are well-defined and continuous and they induce natural transformations $\eta:1_{\textbf{CH}}\to \mathcal{R}$ and $\mu:\mathcal{R}\mathcal{R}\to\mathcal{R}$.

As has often been observed \cite{jacobs, keimel, semadeni,swirszcz}:
\begin{prop}
The triple $(\mathcal{R},\eta,\mu)$ is a monad. \label{radon}
\end{prop}
The monad in Proposition \ref{radon} is called the \textbf{Radon monad}.

Every finite set $A$ endowed with the whole powerset as topology, forms a compact Hausdorff space. Every map of finite sets is continuous with respect to these topologies. This gives a functor $k:\textbf{Set}_f\to \textbf{CH}$.

We will now consider the functor $R$ which is defined as $$\textbf{Set}_f\xrightarrow{k}\textbf{CH}\xrightarrow{\mathcal{R}}\textbf{CH}$$

This means that for a finite set $A$, the space $RA$ is the subspace of $[0,1]^A$ of all families $(p_a)_a$ such that $\sum_{a\in A}p_a=1$. A map of finite sets $f:A\to B$ induces a map continuous $Rf:RA\to RB$ that sends an element $(p_a)_{a\in A}$ to $(\sum_{a\in f^{-1}(b)}p_a)_{b\in B}$. 

\begin{thm}
The Radon monad is the codensity monad of $R$. \label{radonmonad}
\end{thm}
\begin{proof}
By Proposition \ref{Kan} it follows that the codensity monad of $R$ exists. 
Let $X$ be a compact Hausdorff space. We will now show that $$\mathcal{R}(X) = \lim(X\downarrow R\xrightarrow{U}\textbf{Set}_f\xrightarrow{R}\textbf{CH}).$$
Proposition \ref{Kan} then implies that $T^R(X)\cong \mathcal{R}(X)$ for all compact Hausdorff spaces $X$.

For a continuous function $f:X\to RA$, define a map $p_f:\mathcal{R}X\to RA$ by

$$\mathcal{R}X\xrightarrow{\mathcal{R}f}\mathcal{R}\mathcal{R}kA\xrightarrow{\mu_{kA}}\mathcal{R}kA=RA.$$
That means that 
$$p_f(\mathbb{P}):=\left(\int_Xf_a\text{d}\mathbb{P}\right)_{a\in A}.$$
In the same way as in the proof of Theorem \ref{Girymonad} we can show that $(\mathcal{R}X,(p_f)_f)$ is a cone over the diagram $X\downarrow R \xrightarrow{U}\textbf{Set}_f\xrightarrow{R}\textbf{CH}$. 
Let $(Y,(q_f)_f)$ be a cone over the diagram. In the same way as in the proof of Theorem \ref{Girymonad}, we can define a map $I_y:\textbf{CH}(X,[0,1])\to [0,1]$ and show that this map sends $1$ to $1$ and preserves binary sums that exist in $\textbf{CH}(X,[0,1])$. Now by Example \ref{ch}, we know that there exists a unique Radon probability measure $\mathbb{P}_y$ on $X$ such that $I_y(f)=\int_Xf\text{d}\mathbb{P}_y$ for all continuous functions $f:X\to [0,1]$. The map $q:Y\to \mathcal{R}X$ defined by the assignment $y\mapsto \mathbb{P}_y$ is continuous and is the only morphism of cones from $(Y,(q_f)_f)$ to $(\mathcal{R}X, (p_f)_f)$. This shows that $\mathcal{R}(X)\cong T^R(X)$ for all compact Hausdorff spaces $X$. Moreover, this induces a natural isomorphism $\mathcal{R}\cong T^R$. It can be checked that the unit and multiplication of the Radon monad are the same as those of the codensity monad of $R$.
\end{proof}

\begin{rem}
The domain of $R$ is the category of finite sets and maps. We would expect that this would only yield \emph{finitely additive} measures. However Theorem \ref{radonmonad} tells us that the codensity construction gives us Radon probability measures, which are \emph{$\sigma$-additive}. 
\end{rem}
\begin{rem}
Instead of using the category of all finite sets as the domain of $R$ it is enough to use the category of sets $A$ with $\lvert A\rvert \leq 3$.
\end{rem}
\subsection{Bounded Lipschitz monad}\label{blmonad}
In this section we will consider the metric version of section \ref{6.1}. We introduce a new monad of Radon probability measures on the category of compact metric spaces with $1$-Lipschitz maps. We show that this monad is the codensity monad of a functor $L$. This functor is similar to the functor $R$ from section \ref{6.1}.

Let $X$ be a compact metric space and define $\mathcal{L}X$ to be the metric space of all Radon probability measures on $X$ together with the metric $d_{\mathcal{L}X}$ defined by $$d_{\mathcal{L}X}(\mathbb{P}_1,\mathbb{P}_2):=\sup\left\{\left\lvert\int_Xf\text{d}\mathbb{P}_1-\int_Xf\text{d}\mathbb{P}_2\right\rvert\mid f:X\to [0,1] \text{ is a }1\text{-Lipschitz function}\right\}$$
for all Radon probability measures $\mathbb{P}_1$ and $\mathbb{P}_2$ on $X$. We call this metric the \textbf{bounded Lipschitz metric}\footnote{The Wasserstein distance is the metric obtained by taking the supremum over all $1$-Lipschitz function on $X$ taking values in $\mathbb{R}$ instead of $[0,1]$. The obtained metric space is called the Kantorovich space of $X$.}.

Note that this metric space is homeomorphic to $\mathcal{R}X$ as a topological space and therefore it is also compact. Using that the composition of $1$-Lipschitz functions is a $1$-Lipschitz function we can show that pushing forward along a $1$-Lipschitz function $f$ defines a $1$-Lipschitz function $\mathcal{L}f$.

Let $\textbf{KMet}_1$ be the category of compact metric spaces and $1$-Lipschitz maps. The above defines a functor $\mathcal{L}:\textbf{KMet}_1\to \textbf{KMet}_1$.

For a compact metric space $(X,d)$ let $\eta_X$ and $\mu_X$ be the maps as defined for the Radon monad. For $x$ and $y$ in $X$ we find $$d_{\mathcal{L}X}(\eta_X(x),\eta_X(y))=\sup\{\lvert f(x)-f(y)\rvert \mid f:X\to [0,1] \text{is a }1\text{-Lipschitz function}\}\leq d(x,y),$$
because $\lvert f(x)-f(y)\rvert \leq d(x,y)$ for every $1$-Lipschitz function $f:X\to [0,1]$. Therefore $\eta_X$ is $1$-Lipschitz with respect to the bounded Lipschitz metric.

For Radon probability measures $\textbf{P}_1$ and $\textbf{P}_2$ on $\mathcal{L}X$ we find for every $1$-Lipschitz function $f:X\to [0,1]$ the following inequality:
\begin{align*}
    \left\lvert \int_{\mathcal{L}X} f\text{d}\mu_X(\textbf{P}_1) -\int_{\mathcal{L}X}f\text{d}\mu_X(\textbf{P}_2)\right\rvert & = \left\lvert \int_{\mathcal{L}X}\text{ev}_f\text{d}\textbf{P}_1-\int_{\mathcal{L}X}\text{ev}_f\text{d}\textbf{P}_2\right\rvert\\
    &\leq d_{\mathcal{L}\mathcal{L}X}(\textbf{P}_1,\textbf{P}_2)
\end{align*}
In the last step we used the fact that the map $\text{ev}_f:\mathcal{L}X\to [0,1]$, which sends a Radon probability measure $\mathbb{P}$ to $\int_Xf\text{d}\mathbb{P}$, is $1$-Lipschitz. Since $f$ was arbitrary we conclude that $\mu_X$ is $1$-Lipschitz.

It follows now that we have natural transformations $1_{\textbf{KMet}_1}\to \mathcal{L}$ and $\mathcal{L}\mathcal{L}\to \mathcal{L}$ which we will also denote by $\eta$ and $\mu$.

Similar to Proposition \ref{Giry} and Proposition \ref{radon} we can prove the following result.
\begin{prop}
The triple $(\mathcal{L},\eta,\mu)$ is a monad. \label{bl}
\end{prop}

We call the monad from Proposition \ref{bl} the \textbf{bounded Lipschitz monad}. This monad is similar to the Kantorovich monad, which is dicussed in \cite{Fritz1, Fritz2, breugel}. 

Let $A$ be a finite set. Define a metric on $A$ by $$d_A(a_1,a_2):=\begin{cases}1 \quad \text{if }a_1\not=a_2\\
0 \quad \text{if }a_1=a_2\end{cases}$$
for all $a_1$ and $a_2$ in $A$. This makes $A$ a compact metric space. Every function of finite sets becomes $1$-Lipschitz with respect to these metrics. This gives a functor $l:\textbf{Set}_f\to \textbf{KMet}_1$.

We will now consider the functor $L$ which we define as $$\textbf{Set}_f\xrightarrow{l}\textbf{KMet}_1\xrightarrow{L}\textbf{KMet}_1.$$
Let $A$ be a finite set. The space $LA$ is the subset of $[0,1]^A$ of  families $(p_a)_a$ such that $\sum_{a\in A}p_a=1$ together with the bounded Lipschitz metric $d_{LA}$.
\begin{prop}
Let $A$ be a finite set. For $p$ and $q$ in $LA$, $$d_{LA}(p,q)=\sup\left\{\left\lvert \sum_{a\in A'}p_a-\sum_{a\in A'}q_a \right\rvert \mid  A'\subseteq A\right\}.$$
\end{prop}
\begin{proof}
We view $A$ as a metric space by endowing it with the metric $d_A$.

For a subset $A'$ of $A$, the function $1_{A'}:A\to [0,1]$ is a $1$-Lipschitz function. We find $$\left\lvert \sum_{a\in A'}p_a-\sum_{a\in A'}q_a\right\rvert = \left \lvert \int_A1_{A'}\text{d}p-\int_A1_{A'}\text{d}q\right\rvert \leq d_{LA}(p,q).$$
This shows one inequality. 

Define the sets $A^+=\{a\in A\mid p_a\geq q_a\}$ and $A^-=\{a\in A\mid p_a \leq q_a\}$. For a $1$-Lipschitz function $f:A\to [0,1]$, we find 
\begin{align*}
    \int_Af\text{d}p-\int_Af\text{d}q & =\sum_{a\in A}f(a)(p_a-q_a)\\
    & \leq \sum_{a\in A^+}f(a)(p_a-q_a)\\
    & \leq \sum_{a\in A^+}p_a-q_a\\
    & = \left\lvert \sum_{a\in A^+}p_a-\sum_{a\in A^+}q_a\right\rvert
\end{align*}
Similarly we find $$-\left(\int_Af\text{d}p-\int_Af\text{d}q\right) \leq \left\lvert \sum_{a\in A^-}p_a-\sum_{a\in A^-}q_a\right\rvert.$$
We can conclude that $$\left\lvert \int_A f \text{d}p-\int_Af\text{d}q\right\rvert \leq \sup\left\{\left\lvert \sum_{a\in A'}p_a-\sum_{a\in A'}q_a\right\rvert \mid A'\subseteq A\right\},$$
which proves the other inequality.
\end{proof}
This gives us the following useful corollary. 
\begin{cor}
Let $X$ be a compact metric space and let $A$ be a finite set. A map $f:X\to LA$ is $1$-Lipschitz if and only if $$x\mapsto \sum_{a\in A'}f(x)_a$$ defines a $1$-Lipschitz map $X\to [0,1]$ for every $A'\subseteq A$. \label{lips}
\end{cor}
This leads to the main result of this section.
\begin{thm}
The bounded Lipschitz monad is the codensity monad of $L$. \label{blip}
\end{thm}
\begin{proof}
We will again use Proposition \ref{Kan}. Let $(X,d)$ be a compact metric space and let $D_X$ denote the diagram $$X\downarrow L\xrightarrow{U}\textbf{Set}_f\xrightarrow{L}\textbf{KMet}_1.$$
For a $1$-Lipschitz function $f:X\to LA$ define the map $p_f:\mathcal{L}X\to LA$ by 
$$\mathcal{L}X\xrightarrow{\mathcal{L}f}\mathcal{L}\mathcal{L}lA\xrightarrow{\mu_{lA}}\mathcal{L}lA=LA$$
This means that for every $\mathbb{P}\in \mathcal{L}X$,
$$p_f(\mathbb{P})=\left(\int_Xf_a\text{d}\mathbb{P}\right)_{a\in A}.$$

In the same way as in the proof of Theorem \ref{Girymonad}, it can be shown that $(\mathcal{L}X,(p_f)_f)$ forms a cone over the diagram $D_X$.

We will now show that this is the limiting cone over the diagram.

Let $(Y,(q_f)_f)$ be a cone over the diagram $D_X$. For a $1$-Lipschitz function $f:X\to[0,1]$ let $\widehat{f}:X\to L\textbf{2}$ be the function defined by $\widehat{f}(x):=(1-f(x),f(x))$. Note that this function is $1$-Lipschitz. 

For an element $y\in Y$ define the map $I_y:\textbf{KMet}_1(X,[0,1])\to [0,1]$ by $$I_y(f):=q_{\widehat{f}}(y)_1.$$
For $n\geq 1$, let $\textbf{n}:=\{0,\ldots,n-1\}$. Let $t:\textbf{1}\to \textbf{2}$ be the singleton map that sends $0$ to $1$ and let $e$ be the unique $1$-Lipschitz map $X\to L\textbf{1}$. We have that $Lt\circ e = \widehat{1}$, i.e. $t$ is a morphism $\widehat{1}\to e$ in the arrow category $X\downarrow L$. Since $(Y,(q_f)_f)$ is a cone over the diagram $D_X$ it follows that $Lt\circ q_e=q_{\widehat{1}}$. This implies that $I_y(1)=1$.

Consider $1$-Lipschitz functions $f_1,f_2:X\to [0,1]$ such that also $f:=f_1+f_2$ is a $1$-Lipschitz function that takes values in $[0,1]$. Define the map $h:X\to L\textbf{2}$ that sends an element $x$ to $(1-f(x),f_1(x), f_2(x))$. Because $1, f_1,f_2,1-f_2,1-f_1,1-f$ and $f$ are all $1$-Lipschitz, so is $h$ by Corollary \ref{lips}. Let $s:\textbf{3}\to \textbf{2}$ be the singleton map that sends $0$ to $0$ and the other elements to $1$. For an element $k\in \textbf{3}$, let $s_k:\textbf{3}\to \textbf{2}$ be the singleton map that sends $k$ to $1$ and the other elements to $0$.

We have that $$Ls\circ h = \widehat{f}$$ and for every $k\in \textbf{3}$ that 
$$Ls_k\circ h = \widehat{f_k}.$$
Because $(Y,(q_f)_f)$ is a cone over the diagram $D_X$, these equalities imply that $$Ls\circ q_h = q_{\widehat{f}}$$ and for every $k\in \textbf{3}$ that 
$$Ls_k\circ q_h = q_{\widehat{f_k}}.$$
Using this we obtain the following equalities:
$$I_y(f)=(Ls\circ q_{h}(y))_1=q_h(y)_1+q_h(y)_2=Ls_1\circ q_h(y))_1+Ls_2\circ q_h(y))_2+=I_y(f_1)+I_y(f_2).$$

By Example \ref{compactmetric} there exists a unique Radon probability measure $\mathbb{P}_y$ such that $I_y(f)=I_{\mathbb{P}_y}(f)$. The assignment $y\mapsto \mathbb{P}_y$ defines a map $q:Y\to \mathcal{L}X$. For $y_1$ and $y_2$ in $Y$  and for a $1$-Lipschitz function $f:X\to [0,1]$, we find that $$\left\lvert \int_Xf\text{d}\mathbb{P}_{y_1}-\int_Xf\text{d}\mathbb{P}_{y_2}\right\rvert = \lvert I_{y_1}(f)-I_{y_2}(f)\rvert = \lvert q_{\widehat{f}}(y_1)_1-q_{\widehat{f}}(y_2)_1\lvert \leq d_Y(y_1,y_2).$$
It follows now that $q$ is $1$-Lipschitz. 

Let $f:X\to LA$ be a $1$-Lipschitz map and let $a$ be an element of $A$. Let $s_a:A\to \textbf{2}$ be the singleton map that sends $a$ to $1$ and every other element to $0$. Since we have that $Ls_a\circ f = \widehat{f_a}$ we also have that $Ls_a\circ q_f=q_{\widehat{f_a}}$. In particular we find for every $y\in Y$ that $$I_y(f_a)=q_{\widehat{f_a}}(y)_1=(Ls_a\circ q_f(y))_1=q_f(y)_a.$$
Using this we obtain for every $y\in Y$ and for every $a\in A$ that $$p_f\circ q(y)_a=I_{\mathbb{P}_y}(f_a)=I_y(f)_a=q_f(y)_a.$$
This shows that $q$ is a morphism of cones from $(\mathcal{L}X,(p_f)_f)$ to $(Y,(q_f)_f)$. 

Let $\Tilde{q}:Y\to \mathcal{L}X$ be another morphism of cones. Then for every 
$1$-Lipschitz function $f:X\to [0,1]$ we have that
$$\int_Xf\text{d}\Tilde{q}(y)=(p_{\widehat{f}}\circ \Tilde{q}(y))_1=q_{\widehat{f}}(y)_1=I_y(f).$$
This show that $\Tilde{q}(y)=q(y)$ for all $y\in Y$ and therefore $(\mathcal{L}X,(p_f)_f)$ is the limiting cone over the diagram $D_X$. This implies that $\mathcal{L}(X)\cong T^L(X)$ for all compact metric spaces. This induces a natural isomorphism $\mathcal{L}\cong T^L$. 

It can be checked that the unit and multiplication of the codensity monad of $L$ are equal to the unit and multiplication of the bounded Lipschitz monad. 
\end{proof}

\begin{rem}
Note that the bounded Lipschitz monad can be extended to a monad on the category $\textbf{KMet}$ of compact metric spaces and all Lipschitz maps. The compact metric space $\mathcal{L}X$ is isomorphic to the Kantorovich space in this category.
This monad can also be obtained as a codensity monad in a similar way. The proof of this is a simple adaptation of the proof Theorem \ref{blip}.
\end{rem}

\begin{rem}
The Kantorovich monad on compact metric spaces introduced by van Breugel in \cite{breugel} was extended to a monad on complete metric spaces by Fritz and Perrone \cite{Fritz1,Fritz2}. The underlying endofunctor of this monad sends a complete metric space $(X,d)$ to the Kantorovich space of $(X,d)$. This is the space of all Radon probability measure of finite moment together with the Wasserstein distance. 

We would like to use the integral representation theorem from Example \ref{met1fm} to construct this monad as a codensity monad, but we have been unable to find a way to do so.
\end{rem}

\subsection{Baire monad}\label{bairemonad}
In this section we will introduce a new monad of probability measures, which we will call the \emph{Baire monad} and we explain how this monad arises as a codensity monad. We write $\textbf{Haus}$ to mean the category of Hausdorff spaces and continuous maps. 

For a Hausdorff space $X$ let $\mathcal{B}X$ be the space of Baire probability measures on $X$ with the smallest topology such that $\text{ev}_f:\mathcal{B}X\to [0,1]$ is continuous for every continuous function $f:X\to [0,1]$. Much as in the previous sections this induces a functor $\mathcal{B}:\textbf{Haus}\to \textbf{Haus}$ where $\textbf{Haus}$ is the category of Hausdorff spaces and continuous maps.  We can define a monad structure on this endofunctor and we call this monad the \textbf{Baire monad}.

For a countable set $A$ let $BA$ be the subspace of $[0,1]^A$ of families $(p_a)_{a\in A}$ such that $\sum_{a\in A}p_a=1$. Every finite map\footnote{Finite maps were introduced in subsection \ref{4.3}} $f:A\to C$ of countable sets induces a continuous map $Bf:BA\to BC$. Let $\textbf{Set}_c^f$ be the category of countable sets and finite maps. We obtain a functor $B:\textbf{Set}_c^f\to \textbf{Haus}$.

\begin{thm}
The Baire monad is the codensity monad of $B$.\label{baire}
\end{thm}
The proof of this result is similar to the proof of Theorem \ref{premeasmonad} and Theorem \ref{Girymonad}. Therefore we will only give a short sketch of the proof. We will again use Proposition \ref{Kan} and show that $\mathcal{B}X$ is the limit of the diagram $$X\downarrow B \rightarrow \textbf{Set}_c^f\xrightarrow{B}\textbf{Haus}.$$
Every continuous map $f:X\to BA$ induces a map $p_f:\mathcal{B}X\to BA$ by sending $\mathbb{P}$ to $\int_Xf\text{d}\mathbb{P}$. Because we use finite maps and the finite sum of continuous functions is continuous, these maps form a cone over the diagram.

To show that this cone is universal we can use a similar argument as in the proof of Theorem \ref{Girymonad}, since all the maps introduced here were finite maps. Only now we use the representation theorem from Example \ref{eghaus} instead of Proposition \ref{representation}.
\begin{rem}
The results in this section would also work for arbitrary topological spaces. Because it is not common in probability theory to study probability measures on non-Hausdorff spaces we chose to restrict everything to Hausdorff spaces. 
\end{rem}
\begin{rem}
The construction of the Baire monad as a codensity monad is slightly different from the other probability monads. Here the functor is not of the form $$\textbf{Set}_c^f\xrightarrow{i}\textbf{Haus}\xrightarrow{B}\textbf{Haus}.$$
Therefore it does not completely follow the pattern of the constructions in the previous monads. We tried to construct this monad as the codensity monad of a functor of the of form $\mathcal{B}i$ for some (inclusion) functor $i:\textbf{Set}_c^f\to \textbf{Haus}$, but we have not been able to do this so far. 
\end{rem}

\newpage
\begin{appendices}
\section{Integral representation theorems}
\subsection{Carath\'eodory extension theorem}
Let $X$ be a set. \begin{defn}
 A family $S$ of subsets of $X$ is called a \textbf{semi-ring} if it contains $\emptyset$ and is closed under finite intersections and such that the relative complement of a set in $S$ can be written as a finite union of sets in $S$. 
\end{defn} Let $\sigma(S)$ denote the smallest $\sigma$-algebra on $X$ that contains $S$.
\begin{thm}[Carath\'eodory]
Let $\mu:S\to [0,\infty)$ be a map such that $\mathbb{P}(\emptyset)=0$. Suppose that for every collections $(A_n)_{n\in \mathbb{N}}$ of pairwise disjoint elements in $S$ such that their union $A$ is also in $S$ we have $\mu\left(A\right)=\sum_{n\in \mathbb{N}} \mu(A_n)$. Then $\mu$ extends uniquely to a $\sigma$-finite measure $\rho$ on $(X,\sigma(S))$. \label{cara}
\end{thm}
A proof of this result can be found in \cite{dudley} (Proposition 3.2.4).

\subsection{Integration lattices}
Let $X$ be a set and let $f$ and $g$ be real-valued functions on $X$. We use the notation $f\vee g$ ($f\wedge g$) to mean the pointwise maximum (minimum) of the two functions.

Let $L\subseteq [0,\infty)^X$ be a subset of non-negative valued functions on $X$. A subset $L$ is called an \textbf{integration lattice} if $1\in L$ and for all $f,g\in L$ and $r\in [0,\infty)$ also $f\wedge g, f\vee g, rf$ and $f\vee g-f\wedge g$ are elements of $L$.

A map $I:L\to [0,\infty)$ is called an \textbf{integration operator} if $I(1)=1$ and for every collection $(f_n)_n$ in $L$ such that $f:=\sum_{n\in \mathbb{N}} f_n$ is also in $L$ we have $I(f)=\sum_{n\in \mathbb{N}} I(f_n)$. 
\begin{lem}
Let $I$ be a integration operator on a integration lattice $L$ and let $r\in [0,\infty)$. Then $I(rf)=rI(f)$.\label{linear}
\end{lem}
\begin{proof}
For $n,m\geq 1$ and $f\in L$ we find the following:

$$I\left(\frac{m}{n}f\right)=m\left(\frac{1}{n}f\right)=\frac{m}{n}nI\left(\frac{1}{n}f\right)=\frac{m}{n}I(f).$$
It follows that the statement holds for $r\in [0,\infty)\cap \mathbb{Q}$. Now take any $r\in [0,\infty)$ and let $(q_n)_n$ be an increasing sequence of rational non-negative numbers that converges to $r$. We find that $rf=q_1f+\sum_{n\in \mathbb{N}} (q_{n+1}-q_n)f$ and therefore we have that \[I(rf)=I(q_1f)+\sum_{n\in \mathbb{N}} I((q_{n+1}-q_n)f)=q_1+\sum_{n\in \mathbb{N}} (q_{n+1}-q_n)I(f)=rI(f).\]
\end{proof}
\begin{rem}
Lemma \ref{linear} is also true in the case that $I$ is a \emph{weak} integration operator on a \emph{weak} integration lattice $L$ and $r$ is an element of $[0,1]$. \label{linear2}
\end{rem}
The following result is a variant of the Daniell-Stone representation theorem and the proof follows Kindler's proof in \cite{kindler} closely.
\begin{thm}
Let $I$ be an integration operator on an integration lattice $L$. There exists a unique probability measure $\mathbb{P}$ on $(X,\sigma(L))$ such that $$I_\mathbb{P}(f)=I(f)$$ for all $f\in L$. \label{representation1}
\end{thm}
\begin{proof}
For $f,g\in L$ such that $f\leq g$ define $$[f,g):=\{(x,t)\in X\times [0,\infty)\mid f(x)\leq t < g(x)\}$$
and let $S$ be the family of all subsets of $X\times [0,\infty)$ of this form.

For $f_1\leq g_1$ and $f_2\leq g_2$ in $L$ we have:
\[[f_1,g_1)\cap [f_2,g_2)= [f_1\vee f_2,g_1\wedge g_2)\]
and \[[f_1,g_1)\setminus [f_2,g_2)=[f_1,g_1\wedge f_2)\cup [f_1\vee g_2,g_1\wedge (g_1\vee g_2)).\]
Since we also clearly have that $\emptyset \in S$ and $L$ is closed under taking finite minima and maxima, we conclude that $S$ is a semi-ring.

Now define $\mu:S\to [0,\infty)$ by $$\mu([f,g)):=I(g-f)$$ for all $f\leq g$. Note that this is well-defined since $g-f=g\vee f - g\wedge f$, which is an element of $L$. We clearly have that $\mu(\emptyset)=\mu([f,f))=I(0)=0$. For a collection $([f_n,g_n))_{n\in \mathbb{N}}$ of pairwise disjoint subsets in $S$ such that $\bigcup_{n\in \mathbb{N}}[f_n,g_n) = [f,g)$ for some $f\leq g$ in $L$ we can show that $$g-f=\sum_{n\in \mathbb{N}}(g_n-f_n).$$
It follows that \[\mu([f,g))=I(g-f)=\sum_{n\in \mathbb{N}} I(g_n-f_n)=\sum_{n\in \mathbb{N}} \mu([f_n,g_n)).\] It follows now from Theorem \ref{cara} that $\mu$ can be extended uniquely to a $\sigma$-finite measure $\rho$ on $(X\times [0,\infty),\sigma(S))$. 

We will now show that $\sigma(L)\otimes \text{Bo}_{[0,\infty)}\subseteq \sigma(S)$. For $f\in L$ and $r\in [0,\infty)$ define $f_n:=n(f\vee r-f)\wedge 1$ and note that this an element of $L$. For $s\in [0,\infty)$ we have $$\bigcup_{n=1}^\infty [0,sf_n)=\{f< r\}\times [0,s).$$
Therefore every subset of the form $\{f<r\}\times [0,s)$ is in $\sigma(S)$. This shows that $\sigma(L)\otimes \text{Bo}_{[0,\infty)}\subseteq \sigma(S)$. 

Now define $\mathbb{P}:\sigma(L)\to [0,1]$ by $\mathbb{P}(A):=\rho(A\times [0,1))$. Note that for $r$ and $s $ in $[0,\infty)$ we have  
\begin{align*}
    \rho(\{f< r\}\times [0,s))& =\rho\left(\bigcup_{n=1}^\infty [0,sf_n)\right)\\
    & = \lim_{n\to \infty}\rho([0,sf_n))\\
    & =\lim_{n\to \infty}I(sf_n)\\
    &= \lim_{n\to \infty}sI(f_n)\\
    &= s\lim_{n\to \infty}\rho([0,f_n))\\
    & = s \rho(\{f< r\}\times [0,1)) = s\mathbb{P}(\{f< r\})
\end{align*}
Here we used that $[0,sf_n)\subseteq [0,sf_{n+1})$ and Lemma \ref{linear}.
It follows that \[\rho(\{r_1\leq f < r_2\}\times [0,s))=s\mathbb{P}(\{r_1\leq f < r_2\}).\]
We have \[\mathbb{P}(X)=\rho(X\times [0,1))=\rho([0,1))=I(1)=1\] and the $\sigma$-additivity of $\mathbb{P}$ is inherited by $\rho$. Therefore $\mathbb{P}$ is a probability measure on $(X,\sigma(L))$.

For a measurable map $f:X\to [0,\infty)$ there is an increasing sequence of non-negative simple functions $(s_n:=\sum_{k=1}^{m_n}a^n_k1_{A^n_k})_n$, with $A^n_k$ of the form $\{r_1\leq f < r_2\}$, that converges to $f$. Define \[B_n:=\bigcup_{k=1}^{m_n}A_k^n\times [0,a^n_k)\] and note that $\bigcup_{n=1}^\infty B_n=[0,f)$ and that $B_n\subseteq B_{n+1}$ for all $n$. We have the following equalities:

\begin{equation} I(f)=\rho([0,f))=\lim_{n\to \infty}\rho(B_n)\label{x}\end{equation} 
and 
\begin{equation}
\rho(B_n)=\sum_{k=1}^{m_n}\rho(A^n_k\times [0,a^n_k))=\sum_{k=1}^{m_n}a^n_k\mathbb{P}(A^n_k) = I_{\mathbb{P}}(s_n) \label{y}
\end{equation}
Combining \eqref{x} and \eqref{y} with the monotone convergence theorem we conclude that $I(f)=I_{\mathbb{P}}(f)$.

Let $\mathbb{P}'$ be another probability measure with this property. Then we have for all $f\in L$ and $r\in[0,\infty)$ that $$\mathbb{P}'(\{f< r\})=\lim_{n\to \infty}I_\mathbb{P}'(f_n)=\lim_{n\to \infty}I_\mathbb{P}(f_n)=\mathbb{P}(\{f<r\}).$$
Since the sets $\{f<r\}$ form a $\pi$-system that generates $\sigma(L)$ we can conclude that $\mathbb{P}=\mathbb{P}'$. 
\end{proof}
\subsection{Proofs of the results in section \ref{5}}
\begin{proof}[Proof of Theorem \ref{representation2}]
It is easy to verify that $\mathbb{N}L$ is an integration lattice. For an element $mf\in \mathbb{N}L$ define $I'(mf):=mI(f)$. Suppose $mf=ng$ for $f,g\in L$. Then we have by Remark \ref{linear2} that $$mI(f)=(m\vee n)\frac{m}{m\vee n}I(f)=(m\vee n)I\left(\frac{m}{m\vee n}f\right)=(m\vee n)I\left(\frac{n}{m\vee n}g\right)=nI(g).$$
This shows that $I'$ is well-defined and that $I'(f)=I(f)$ for all $f\in L$.

It is clear that $I(1)=1$. Let $(m_nf_n)_{n\in \mathbb{N}}$ be a collection of elements in $\mathbb{N}L$ such that $mf:=\sum_{n\in \mathbb{N}} m_nf_n$ is also an element of $\mathbb{N}L$. We observe that $f=\sum_{n\in \mathbb{N}} \sum_{k\in \mathbb{N}} a(k,n)\frac{f_n}{m}$ where $a(k,n)=1$ if $k\leq m_n$ and $0$ otherwise. Note that $a(k,n)\frac{f_n}{m}$ is an element of $L$ and therefore $$I(f)=\sum_{n\in \mathbb{N}}\sum_{k\in \mathbb{N}} I\left(\frac{a(k,n)}{m}f\right)= \frac{1}{m}\sum_{n\in \mathbb{N}}\sum_{k\in \mathbb{N}} a(k,n)I\left(f\right)=\frac{1}{m}\sum_{n\in \mathbb{N}} m_nI(f_n).$$
We can conclude that $I'$ is an integration operator on the integration lattice $\mathbb{N}L$ and therefore by Theorem \ref{representation1} there is a unique probability measure on $(X,\sigma(\mathbb{N}L))$ such that $\int_Xf\text{d}\mathbb{P}=I'(f)$. It is clear that $\sigma(L)\subseteq\sigma(\mathbb{N}L)$. For $mf\in \mathbb{N}L$ and $r\in [0,\infty)$ we find that $\{mf<r\}=\{f<\frac{r}{m}\}\in \sigma(L)$ and therefore also $\sigma(\mathbb{N}L)\subseteq \sigma(L)$. 

Let $\mathbb{P}'$ be another probability measure on $(X,\sigma(L))$ such that $\int_Xf\text{d}\mathbb{P}' = I(f)$ for all $f\in L$, then we have that $$I'(mf)=mI(f)=m\int_Xf\text{d}\mathbb{P}' =  \int_Xmf\text{d}\mathbb{P}'.$$
This implies that $\mathbb{P}'=\mathbb{P}$. 
\end{proof}
\begin{lem}\label{orderpres}
Let $I$ be a weak integration operator on a weak integration lattice $L$. For $f,g\in L$ such that $f\leq g$ we have $I(f)\leq I(g)$, and if $g-f\in L$ then $I(g-f)=I(g)-I(f)$.
\end{lem}
\begin{proof}
There exists and $n$ such that $\frac{g-f}{n}\in L$. We find \[\frac{1}{n}I(g)=I\left(\frac{g}{n}\right) = I\left(\frac{g-f}{n}+\frac{f}{n} \right)= I\left(\frac{g-f}{n}\right) + \frac{1}{n}I(f)\geq \frac{1}{n}I(f). \]
Here we use Remark \ref{linear2}. The first part of the statement now follows. The second part can be proven in a similar way.
\end{proof}
\begin{proof}[Proof of Proposition \ref{Dini}]
First note that for $f\in L$ and $n\geq 1$, then \begin{align*}
    I(f)&=I\left(\frac{(n-1)f}{n}+\frac{f}{n}\right)\\
    & = I\left(\frac{(n-1)f}{n}\right)+I\left(\frac{f}{n}\right)\\
    & =I\left(\frac{(n-2)f}{n}\right)+I\left(\frac{f}{n}\right)+I\left(\frac{f}{n}\right)\\
    & = \ldots \\
    & = n I\left(\frac{f}{n}\right).
\end{align*}
By Theorem \ref{representation2} it is enough to show that $I$ is a weak integration operator. Let $(f_n)_{n\in \mathbb{N}}$ be a collection of elements in $L$ such that also $f:=\sum_{n\in \mathbb{N}}f_n$ is an element of $L$. Define for every $n$ the function $g_n:=\sum_{k\leq n}f_k$. For every $n$ there exists an $m_n$ such that $\frac{g_k}{m_n}\in L$ for all $k\leq n$. Indeed, for $n=1$ this is clear. Suppose the claim holds for a natural number $n$, then $\frac{g_n}{m_n}$ and $\frac{f_{n+1}}{m_n}$ are element of $L$. Because $$\frac{g_{n+1}}{m_n}=\frac{g_n}{m_n}+\frac{f_{n+1}}{m_n}$$ is an element of $\mathbb{N}L$, there exists an $m'$ such that $\frac{1}{m'}\frac{g_{n+1}}{m_n}\in L$. Since $\frac{g_k}{m_n}\in L$ for $k\leq n$, we also have that $\frac{g_k}{m'm_n}\in L$. It follows that for $m_{n+1}:=m'm_n$  we have that $\frac{g_k}{m_{n+1}}\in L$ for all $k\leq n+1$. The claim now follows by induction.

The functions $(g_n)_n$ form a sequence of continuous functions that increases pointwise to the continuous function $f$. Because $X$ is a compact Hausdorff space, Dini's theorem tells us that $(g_n)_n$ converges uniformly to $f$. 

Let $\epsilon>0$. There exists an $n$ such that 
\begin{equation}\label{appendixeq}
    \lVert f-g_n\rVert \leq \epsilon.
\end{equation}

By the above there exists an $m_n$ such that $\frac{1}{m_n}g_k\in L$ for all $k\leq n$. Because $L$ is a weak integration lattice, there exists an $m$ such that \[h_n:=\frac{1}{m}\left(\frac{f}{m_n}-\frac{g_n}{m_n}\right)\]
is also an element of $L$. By \eqref{appendixeq} and the first part of Lemma \ref{orderpres} we find that \begin{equation}\label{a}I(h_n)\leq I\left(\frac{\epsilon}{m_nm}\right)=\frac{\epsilon}{m_nm}.\end{equation} By the second part of Lemma \ref{orderpres} we have \begin{equation}\label{b}
    I\left(\frac{f}{m_nm}\right)-I\left(\frac{g_n}{m_nm}\right)=I(h_n). \end{equation}
Using \eqref{a} and \eqref{b} and the hypothesis we find the following:
\begin{align*}
    I(f)-\sum_{k\leq n}I(f_k) & = I\left(f\right)-\sum_{k\leq n}m_nI\left(\frac{f_k}{m_n}\right) \\
    & = I(f)-m_n\left(I\left(\frac{f_1}{m_n}\right)+I\left(\frac{f_2}{m_n}\right)+I\left(\frac{f_3}{m_n}\right)+\ldots +I\left(\frac{f_n}{m_n}\right)\right)\\
    & = I(f)-m_n\left(I\left(\frac{g_2}{m_n}\right)+I\left(\frac{f_3}{m_n}\right)+\ldots +I\left(\frac{f_n}{m_n}\right)\right)\\
    &= I(f)-m_n\left(I\left(\frac{g_3}{m_n}\right)+\ldots +I\left(\frac{f_n}{m_n}\right)\right)\\
    & = \ldots \\
    & = I(f)- m_n I\left(\frac{g_n}{m_n}\right)\\
    & = m_nm\left(I\left(\frac{f}{m_nm}\right)-I\left(\frac{g_n}{m_nm}\right) \right)\\
    &=m_nmI(h_n)<\epsilon
\end{align*}
Because $(I(f)-\sum_{k\leq m}I(f_k))_m$ is a decreasing sequence of positive numbers, we can conclude that $I(f)=\sum_{n\in \mathbb{N}}I(f_n)$.
\end{proof}
\end{appendices}
\newpage
\bibliographystyle{abbrv}
%\nocite{*}
\bibliography{ref.bib}

\begin{thebibliography}{10}

\bibitem{Avery}
T.~Avery.
\newblock Codensity and the {G}iry monad.
\newblock {\em J. Pure Appl. Algebra}, 220(3):1229--1251, 2016.

\bibitem{Rao}
K.~P.~S. Bhaskara~Rao and M.~Bhaskara~Rao.
\newblock {\em Theory of Charges, A Study of Finitely Additive Measures},
  volume 109 of {\em Pure and Applied Mathematics}.
\newblock Academic Press, 1983.

\bibitem{borceuxI}
F.~Borceux.
\newblock {\em Handbook of Categorical Algebra. 1}, volume~50 of {\em
  Encyclopedia of Mathematics and its Applications}.
\newblock Cambridge University Press, Cambridge, 1994.
\newblock Basic category theory.

\bibitem{dudley}
R.~M. Dudley.
\newblock {\em Real Analysis and Probability}, volume~74 of {\em Cambridge
  Studies in Advanced Mathematics}.
\newblock Cambridge University Press, Cambridge, 2002.
\newblock Revised reprint of the 1989 original.

\bibitem{Fritz1}
T.~Fritz and P.~Perrone.
\newblock A probability monad as the colimit of spaces of finite samples.
\newblock {\em Theory Appl. Categ.}, 34:Paper No. 7, 170--220, 2019.

\bibitem{Fritz2}
T.~Fritz and P.~Perrone.
\newblock Stochastic order on metric spaces and the ordered {K}antorovich
  monad.
\newblock {\em Adv. Math.}, 366:107081, 46, 2020.

\bibitem{giry}
M.~Giry.
\newblock A categorical approach to probability theory.
\newblock In {\em Categorical aspects of topology and analysis ({O}ttawa,
  {O}nt., 1980)}, volume 915 of {\em Lecture Notes in Math.}, pages 68--85.
  Springer, Berlin-New York, 1982.

\bibitem{jacobs}
B.~Jacobs.
\newblock From probability monads to commutative effectuses.
\newblock {\em J. Log. Algebr. Methods Program.}, 94:200--237, 2018.

\bibitem{keimel}
K.~Keimel.
\newblock The monad of probability measures over compact ordered spaces and its
  {E}ilenberg-{M}oore algebras.
\newblock {\em Topology Appl.}, 156(2):227--239, 2008.

\bibitem{Kennison}
J.~F. Kennison and D.~Gildenhuys.
\newblock Equational completion, model induced triples and pro-objects.
\newblock {\em J. Pure Appl. Algebra}, 1(4):317--346, 1971.

\bibitem{kindler}
J.~Kindler.
\newblock A simple proof of the {D}aniell-{S}tone representation theorem.
\newblock {\em Amer. Math. Monthly}, 90(6):396--397, 1983.

\bibitem{leinster}
T.~Leinster.
\newblock Codensity and the ultrafilter monad.
\newblock {\em Theory Appl. Categ.}, 28:No. 13, 332--370, 2013.

\bibitem{rudin}
W.~Rudin.
\newblock {\em Functional Analysis}.
\newblock International Series in Pure and Applied Mathematics. McGraw-Hill,
  Inc., New York, second edition, 1991.

\bibitem{semadeni}
Z.~Semadeni.
\newblock {\em Monads and their {E}ilenberg-{M}oore algebras in functional
  analysis}.
\newblock Queen's Papers in Pure and Applied Mathematics, No. 33. Queen's
  University, Kingston, Ont., 1973.

\bibitem{swirszcz}
T.~\'{S}wirszcz.
\newblock Monadic functors and convexity.
\newblock {\em Bull. Acad. Polon. Sci. S\'{e}r. Sci. Math. Astronom. Phys.},
  22:39--42, 1974.

\bibitem{breugel}
F.~van Breugel.
\newblock The metric monad of probabilistic nondeterminism.
\newblock
  \href{http://www.cse.yorku.ca/~franck/research/drafts/monad.pdf}{pdf}, 2005.

\bibitem{sipos}
J.~\v{S}ipo\v{s}.
\newblock Integral with respect to a pre-measure.
\newblock {\em Math. Slovaca}, 29(2):141--155, 1979.

\end{thebibliography}
\end{document}